\documentclass[preprint, 12pt]{elsarticle}

\usepackage{graphicx}
\usepackage{placeins}
\usepackage{subcaption}

\usepackage{amssymb}
\usepackage{amsmath}
\usepackage{amsthm}

\usepackage{hyperref}
\usepackage[usenames,dvipsnames]{xcolor}

\makeatletter
\def\ps@pprintTitle{%
     \let\@oddhead\@empty
     \let\@evenhead\@empty
     \def\@oddfoot{\footnotesize\itshape
     {\copyright 2018. This manuscript version is made available under the
      \href{https://creativecommons.org/licenses/by-nc-nd/4.0/}{CC-BY-NC-ND 4.0 license}.}
      \hfill}
     \let\@evenfoot\@oddfoot}

\usepackage{bm}
\usepackage{accents} 

\newcommand{\eq}[1]{(\ref{#1})}

\def\bmi#1{\textbf{\textit{#1}}}

\newcommand{\wtilde}[1]{\widetilde{#1}}

\newcommand{\ol}[1]{\overline{#1}}
\def\dvg{\mbox{\rm div}}

\def\psibf{{\mbox{\boldmath$\psi$\unboldmath}}}

\def\sigmabf{{\mbox{\boldmath$\sigma$\unboldmath}}}
\def\alphabf{{\mbox{\boldmath$\alpha$\unboldmath}}}
\def\omegabf{{\mbox{\boldmath$\omega$\unboldmath}}}

\def\eeby#1{\eb_y({#1})}

\def\R{\hbox{\rm I\kern-0.2em R}}
\def\Z{\hbox{\rm Z\kern-0.3em Z}}
\def\chii{\mbox{$\chi$ \kern-1.0em $=$}}

\def\pd{\partial}
\def\Om{\Omega}
\def\1mY{\frac{1}{|Y|}}

\def\RR{{\mathbb{R}}}

\def\intY{\sim \kern-1.2em \int}

\def\Om{\Omega}

\def\pd{\partial}

\newcommand\dSy{\mathrm{\,dS}_{y}}  %

\def\1mY{\frac{1}{|Y|}}

\def\RR{{\mathbb{R}}}

\def\thetabf{{\mbox{\boldmath$\theta$\unboldmath}}}

\def\ub{{\bmi{u}}}
\def\vb{{\bmi{v}}}
\def\wb{{\bmi{w}}}
\def\fb{{\bmi{f}}}

\def\nb{{\bmi{n}}}
\def\zb{{\bmi{z}}}
\def\sb{{\bmi{s}}}
\def\tb{{\bmi{t}}}
\def\eb{{\bmi{e}}}
\def\xb{{\bmi{x}}}
\def\gb{{\bmi{g}}}

\def\Bb{{\bmi{B}}}
\def\Cb{{\bmi{C}}}

\def\Rb{{\bmi{R}}}
\def\Sb{{\bmi{S}}}

\def\Wb{{\bmi{W}}}

\def\Kb{{\bmi{K}}}
\def\Rb{{\bmi{R}}}
\def\Ib{{\bmi{I}}}

\def\Pibf{{\mbox{\boldmath$\Pi$\unboldmath}}}

\def\Hdb{{\bf{H}}^1}

\def\Hpdb{{\bf{H}}_\#^1}

\def\Vcal{\mathcal{V}}
\def\Dcal{\mathcal{D}}
\def\Lcal{\mathcal{L}}
\def\Acal{\mathcal{A}}
\def\Bcal{\mathcal{B}}

\def\Pcal{\mathcal{P}}
\def\Fcal{\mathcal{F}}
\def\Ical{\mathcal{I}}
\def\Mcal{\mathcal{M}}

\def\DDcalbf{{\mbox{\mathcal{I} \kern-0.2em\boldmath$\mathcal{D}$\unboldmath}}}

\def\dlt{\delta}

\def\intY{\sim \kern-1.2em \int}
\def\intYs{\smallsim \kern-.75em \int}
\newcommand{\smallsim}{\ensuremath \raisebox{.15em}{{$\scriptstyle\sim$}}}

\def\Dop{{{\rm I} \kern-0.2em{\rm D}}}
\def\Cop{{{\rm C} \kern-0.6em{\rm C}}}
\def\Aop{{{\rm A} \kern-0.6em{\rm A}}}%
\def\Hop{{{\rm I} \kern-0.2em{\rm H}}}%
\def\Iop{{{\rm I} \kern-0.2em{\rm I}}}%

\def\DDcalbf{{\mbox{{\boldmath$\Ical$\unboldmath}\kern-0.5em{\boldmath$\Dcal$\unboldmath}}}}

\newcommand\RELEASED[1]{}

\def\eq#1{(\ref{#1})}

\def\bmi#1{\textbf{\textit{#1}}}
\def\ol#1{\overline{#1}}
\def\dvg{\mbox{\rm div}}
\DeclareMathOperator{\conv}{conv}
\DeclareMathOperator{\coni}{coni}

\def\alphabf{{\mbox{\boldmath$\alpha$\unboldmath}}}
\def\betabf{{\mbox{\boldmath$\beta$\unboldmath}}}
\def\psibf{{\mbox{\boldmath$\psi$\unboldmath}}}
\def\sigmabf{{\mbox{\boldmath$\sigma$\unboldmath}}}
\def\thetabf{{\mbox{\boldmath$\theta$\unboldmath}}}
\def\omegabf{{\mbox{\boldmath$\omega$\unboldmath}}}

\def\pd{\partial}
\def\Om{\Omega}
\def\RR{{\mathbb{R}}}
\def\intY{\sim \kern-1.2em \int}

\def\xb{{\bmi{x}}}
\def\tb{{\bmi{t}}}
\def\ub{{\bmi{u}}}
\def\vb{{\bmi{v}}}
\def\wb{{\bmi{w}}}
\def\fb{{\bmi{f}}}
\def\nb{{\bmi{n}}}
\def\eb{{\bmi{e}}}
\def\gb{{\bmi{g}}}
\def\Bb{{\bmi{B}}}
\def\Cb{{\bmi{C}}}

\def\Wb{{\bmi{W}}}

\def\Kb{{\bmi{K}}}
\def\Ib{{\bmi{I}}}
\def\Pb{{\bmi{P}}}

\def\imx{{\bmi{i}}}

\def\Pibf{{\mbox{\boldmath$\Pi$\unboldmath}}}

\def\Hdb{{\bf{H}}^1}
\def\Hpdb{{\bf{H}}_\#^1}

\def\Vcal{\mathcal{V}}
\def\Mcal{\mathcal{M}}
\def\Acal{\mathcal{A}}
\def\Bcal{\mathcal{B}}

\def\Dop{{{\rm I} \kern-0.2em{\rm D}}}
\def\Hop{{{\rm I} \kern-0.2em{\rm H}}}
\def\Aop{{{\rm A} \kern-0.6em{\rm A}}}%

\def\chE#1{#1} 
\def\dSy{\mathrm{\,dS}_{y}}  %
\def\mx{{[m]}}
\def\eeb#1{\eb({#1})}
\def\eeby#1{\eb_y({#1})}

\def\aYm#1#2{a_{Y}^m \left ({#1},\,{#2}\right )}

\newcommand\ie{{\it{i.e.~}}}
\newcommand\cf{{{cf.~}}}
\newcommand\rhs{{{\rm r.h.s.~}}}
\newcommand\wrt{{\rm with respect to~}}
\newcommand\tot{{\rm tot}}
\newcommand\supp[1]{{\rm supp}\left({#1}\right)}

\def\aOm#1#2{a_{\Om} \left ({#1},\,{#2}\right )}
\def\bOm#1#2{b_{\Om} \left ({#1},\,{#2}\right )}
\def\cOm#1#2{c_{\Om} \left ({#1},\,{#2}\right )}

\newtheorem{theorem}{Theorem}[section]
\newtheorem{lemma}[theorem]{Lemma}
\newtheorem{e-proposition}[theorem]{Proposition}
\newtheorem{corollary}[theorem]{Corollary}
\newtheorem{e-definition}[theorem]{Definition\rm}
\newtheorem{remark}{\it Remark\/}

\graphicspath{.}

\begin{document}

\begin{frontmatter}



\title{Optimization of the porous material described by the Biot model}





\author[FAU]{D.~H\"ubner}
\ead{daniel.huebner@fau.de}
\author[NTIS]{E.~Rohan\corref{cor1}}
\ead{rohan@kme.zcu.cz}
\cortext[cor1]{Corresponding author}
\author[NTIS]{V.~Luke\v{s}}
\ead{vlukes@kme.zcu.cz}
\author[FAU]{M.~Stingl}
\ead{stingl@am.uni-erlangen.de}

\address[NTIS]{Department of Mechanics \& 
  NTIS New Technologies for Information Society, Faculty of Applied Sciences, University of West Bohemia in Pilsen, \\
Univerzitn\'\i~22, 30614 Plze\v{n}, Czech Republic}
\address[FAU]{Department Mathematik, Friedrich-Alexander University Erlangen-N\"urnberg, Cauerstrasse 11, 91058 Erlangen, Germany}


\begin{abstract}
The paper is devoted to the shape optimization of microstructures generating porous locally periodic materials saturated by viscous fluids. At the macroscopic level, the porous material is described by the Biot model defined in terms of the effective medium coefficients, involving the drained skeleton elasticity, the Biot stress coupling, the Biot compressibility coefficients, and by the hydraulic permeability of the Darcy flow model. By virtue of the homogenization, these coefficients are computed using characteristic responses of the representative unit cell consisting of an elastic solid skeleton and a viscous pore fluid.
For the purpose of optimization, the sensitivity analysis on the continuous level of the problem is derived.  We provide sensitivities of objective functions constituted by the Biot model coefficients  with respect to the underlying pore shape described by a
B-spline box which embeds the whole representative cell. We consider material design problems in the framework of which the layout of a single representative cell is optimized. Then we propose a sequential linearization approach to the two-scale problem in which local microstructures are optimized with respect to macroscopic design criteria. Numerical experiments are reported which include stiffness maximization with constraints allowing  for a sufficient permeability, and vice versa. Issues of the design anisotropy, the spline box parametrization are discussed. In order to avoid remeshing a geometric regularization technique based on injectivity constraints is applied.
\end{abstract}


\end{frontmatter}


\section{Introduction}\label{sec-Introduction}

The paper deals with shape optimization of microstructures generating porous locally periodic materials saturated by viscous fluids.
This topic is challenging in the context of new material design with obvious applications in many fields of engineering.
The aim is to design optimal microstructures according to criteria related to the effective material properties.
Potentially this enables to design optimal structures with nonuniform material properties at the macroscopic level, such that a desired functionality of the whole structure is achieved.

A comprehensive overview of shape optimization with the homogenization method is provided by Allaire in \cite{Allaire2001}.
For a general summary of the homogenization method we refer to the first part of the review paper of Hassani and Hinton \cite{Hassani1998} and references therein.
Shape optimization for periodic problems has been conducted also by Barbarosie and Toader \cite{Barbarosie2009}.
The authors deduce the topological and, using the perturbation of identity method, also the shape derivative for elliptic problems in unbounded domains subject to periodicity conditions.
This allows them to apply gradient-based methods to design problems involving homogenized material properties.
In \cite{Barbarosie2012} this work is extended to design problems on two scales with locally periodic microstructures in linear elasticity.

In the context of fluid-saturated media, the problem of the material design is typically treated by the topology optimization methods \cite{Guest-Prevost-2006,Xu-Cheng-2010,Andreasen2013}.
To solve the inverse homogenization problem providing optimized microstructures featured by desired effective material properties relevant to the macroscopic scale, in \cite{Guest-Prevost-2006} the fluid-structure interaction problem is considered as decomposed into two problems solved independently in the solid and fluid parts.
This enables to employ the SIMP method of \cite{Bendsoe2003} as the tool for the topology optimization with a two-criteria objective function based on the compliance and permeability.
The same approach based on the ``design material density'' related to the two phases was pursued in \cite{Xu-Cheng-2010}, where the seepage constitutes the constraint for the maximum stiffness design.
Therein the sensitivity analysis was derived which enabled to use the method of moving asymptotes in the numerical optimization.
A fairly two-scale optimization problem was reported by Andreasen and Sigmund in \cite{Andreasen2013}, where the homogenization approach to the topology optimization was employed to solve the fluid-structure interaction at the pore-level of the Biot type medium \cite{Biot1955} filling a given domain. It is noted that in \cite{Andreasen2013} neither the shape nor the topology of the base cell is optimized. Instead, a given layout is homogenized for different porosity values on a fixed grid and then macroscopic Biot coefficients are calculated by interpolation.

In this article, 
following a number of works dealing with homogenization of the quasistatic fluid-structure interaction in the microstructure, see e.g. \cite{Auriault-Palencia77,Burridge-Keller-1981,Hornung1997book,Lee-Mei-1997,rohan-etal-CMAT2015-porel,RSW-ComGeo2013},
the porous material is described as the Biot continuum derived by the homogenization of two decoupled problems:
1) deformation of a porous solid saturated by a slightly compressible static fluid and 2) Stokes flow through the rigid porous structure.
The effective medium properties are given by the drained skeleton elasticity, the Biot stress coupling, the Biot compressibility coefficients, and by the hydraulic permeability of the Darcy flow model.
These are computed using characteristic responses of the representative unit cell constituted by an elastic skeleton and by the fluid channel.
We consider two different kinds of optimal design problems: The first one is related to material design whereby two criteria of optimal poroelastic material are considered.
First, the objective is to maximize stiffness of the drained porous material and allow for a sufficient permeability and vice versa.
In both cases, anisotropy can be enforced according chosen preferential directions.
The second one involves all the Biot model coefficients: the objective is to minimize the undrained material compliance.
The second kind of the optimum design is related to a given macroscopic domain occupied by the Biot type porous medium.
As the optimality criterion we consider the structure stiffness under a constraint associated with a prescribed seepage through the domain.
Although we do not solve numerically the whole two-scale problem, we present a methodology and give numerical illustrations. As mentioned above, a similar kind of the two-scale optimization problems were considered within the topology optimization framework in \cite{Andreasen2013} and \cite{Xu-Cheng-2010}.

Here, using a spline parametrization of the representative periodic cell, we introduce design variables which describe the pore shape of the locally periodic  microstructures considered in both the kinds of optimum design problems.
Thus, as a difference with the above works dealing with topology optimization, we apply shape optimization to design the base cell and use the shape sensitivity technique based on the material derivative approach
\cite{haslinger_makinen_2003:_introduction_shape_optimization,haug_choi_komkov_1986:_design_sensitivity_analysis}.
In particular, we present shape derivatives of the homogenized coefficients obtained by the velocity or also speed method \cite{Delfour2011}.
Although the modeling as well as the sensitivity analysis is carried out completely on the continuous level, see \cite{Rohan-AMC}, the topology of the computational grid used in the discretized version of the problems is fixed.
A clear advantage of this classic technique (see, e.g.~\cite{haslinger_makinen_2003:_introduction_shape_optimization}) is that remeshing is avoided and that gradients can be computed consistently on the discretized level of the problem. On the other hand shape variations have to be limited in order to prevent degeneration of the finite element mesh, which often results in rather conservative designs.
As a compromise we suggest a spline box approach, in the framework of which the design freedom (i.e. the admissible positions of the control points of the spline) is not limited in an absolute manner. Rather than this, a technique from Xu et. al. \cite{Xu2013} is adapted guaranteeing that for all feasible choices of control points the deformation of the base cell is described by a one-to-one map.
It is shown that sufficient conditions can be expressed by linear functions of the design variables.
This is a clear advantage, as linear constraints can be strictly satisfied during all optimization iterations avoiding break-down of the optimization process due to numerical reasons.


The remainder of the article is structured as follows:
we first introduce the homogenized Biot-Darcy model in its weak formulation and state the coupled model of poroelastic media in Section \ref{sec-linBiot}.
Next, in Section \ref{sec-sa} we derive shape sensitivities of the homogenized coefficients.
In Section \ref{sec-optprob} we give the optimization problem with a detailed description of the parametrization and some regularity preserving constraints to avoid remeshing.
This is followed by some numerical examples in Section \ref{sec-numres}.
Then, in Section \ref{sec-2scopt} we investigate a two-scale problem and optimize the microscopic topology at some exemplary macroscopic locations for a given macroscopic state.
Finally we summarize our findings and give an outlook in Section \ref{sec-concl}.

\section{Linear homogenized Biot continuum}\label{sec-linBiot}

It is well known that within the small strain theory, behavior of the fluid-saturated porous materials is governed
by the Biot model \cite{Biot1955}. The poroelastic coefficients obtained by M.A. Biot hold also for quasistatic situations, see \cite{Burridge-Keller-1981}.
Assuming slow flows through a slightly deforming porous structure
undergoing quasistatic loading of the solid phase by external forces, whereby the inertia
forces can be neglected, the upscaled medium is described by the Biot-Darcy coupled system of equations which was  derived by the homogenization of two decoupled problems: 1)
deformation of a porous solid saturated by a slightly compressible static fluid and 2) Stokes flow through the rigid porous structure, \cf \cite{Mik-Wheel-2012,RSW-ComGeo2013,Rohan-AMC}. In a more general setting, homogenization of the viscous flow in deforming media was considered in \cite{IlievMikelic2008}.

\subsection{The homogenized Biot -- Darcy model}
Here we summarize  the homogenized model of the porous elastic medium. The model incorporates local problems for characteristic responses which are employed to compute the effective material coefficients of the Biot model.

The local problems (specified below) related to the homogenized model
are defined at the representative unit microscopic cell $Y =
\Pi_{i=1}^3]0,\ell_i[ \subset \RR^3$ which splits into the solid part
occupying domain $Y_m$ and the complementary channel part $Y_c$, thus
\begin{equation}\label{eq-6}
\begin{split}
Y   = Y_m  \cup Y_c  \cup \Gamma_Y \;,\quad
Y_c   = Y \setminus Y_m  \;,\quad
\Gamma_Y   = \ol{Y_m } \cap \ol{Y_c }\;,
\end{split}
\end{equation}
where by $\ol{Y_d }$ we denote the closure of the open bounded domain $Y_d$.

We shall employ the following notation. By $\intYs_{Y_d} = |Y|^{-1}\int_{Y_d}$ with
$Y_d\subset \ol{Y}$ we denote the local average, although $|Y|=1$ can
always be chosen. The usual elasticity bilinear
form is
\begin{equation}
\aYm{\wb}{\vb} = \intY_{Y_m} (\Dop \eeby{\wb}) : \eeby{\vb}\;,
\end{equation}
where $\Dop = (D_{ijkl})$ is the elasticity tensor satisfying the usual symmetries, $D_{ijkl} = D_{klij} = D_{jikl}$, and $\eeby{\vb} = \frac{1}{2}(\nabla_y \vb + (\nabla_y \vb)^T)$ is the linear strain tensor associated with the displacement field $\vb$.
The Lebesgue spaces of 2nd-power integrable functions on
an open bounded domain $D\subset \RR^3$ is denoted by $L^2(D)$, the Sobolev space $\Wb^{1,2}(D)$
of the square integrable vector-valued functions on $D$ including the 1st order
generalized derivative, is abbreviated by $\Hdb(D)$.  Further
$\Hpdb(Y_m)$ is the Sobolev space of vector-valued Y-periodic
functions (the subscript $\#$).    \chE{Below we employ also $\Pibf^{ij} = (\Pi_k^{ij})$, $i,j,k   = 1,2,3$ with components $\Pi_k^{ij} = y_j\delta_{ik}$}.

If the structure is perfectly periodic, the decomposition of the
microstructure and the microstructure parameters are independent of
the macroscopic position $x \in \Om$. Otherwise the local problems
must be considered at any macroscopic position, i.e. for almost any
$x \in \Om$, \cf \cite{Brown2011}.

The local microstructural response is obtained by solving the following decoupled problems:
\begin{itemize}
\item Find
$\omegabf^{ij}\in \Hpdb(Y_m)$ for any $i,j = 1,2,3$
satisfying
\begin{equation}\label{eq-h5a}
\begin{split}
\aYm{\omegabf^{ij} + \Pibf^{ij}}{\vb} & = 0\;, \quad \forall \vb \in  \Hpdb(Y_m)\;.
\end{split}
\end{equation}
\item Find
$\omegabf^P \in \Hpdb(Y_m)$
satisfying
\begin{equation}\label{eq-h5b}
\begin{split}
\aYm{\omegabf^P}{\vb} & = \intY_{\Gamma_Y} \vb\cdot \nb^\mx \dSy\;, \quad
\forall \vb \in  \Hpdb(Y_m) \;.
\end{split}
\end{equation}
\item Find $(\psibf^i,\pi^i) \in \Hpdb(Y_c) \times L^2(Y_c)$ for $i = 1,2,3$ such that
\begin{equation}\label{eq-S3}
\begin{split}
\int_{Y_c} \nabla_y \psibf^k: \nabla_y \vb - \int_{Y_c} \pi^k \nabla\cdot \vb & = \int_{Y_c} v_k\quad \forall \vb \in \Hpdb(Y_c) \;,\\
\int_{Y_c} q \nabla_y \cdot \psibf^k & = 0\quad \forall q \in L^2(Y_c)\;.
\end{split}
\end{equation}
\end{itemize}

Using the characteristic responses \eq{eq-h5a}--\eq{eq-S3} obtained at
the microscopic scale the homogenized coefficients, describing the
effective properties of the deformable porous medium, are given by the
following expressions:
\begin{equation}\label{eq-h8}
\begin{split}
A_{ijkl} = \aYm{\omegabf^{ij} + \Pibf^{ij}}{\omegabf^{kl} + \Pibf^{kl}}\;,\quad
C_{ij} =  -\intY_{Y_m}\dvg_y \omegabf^{ij} = \aYm{\omegabf^P}{\Pibf^{ij}}\;,\\
N  =  \aYm{\omegabf^P}{\omegabf^P} = \intY_{\Gamma_Y} \omegabf^P\cdot \nb \dSy\;,\quad
K_{ij}  = \intY_{Y_c} \psi_i^j =  \intY_{Y_c} \nabla_y \psibf^i : \nabla_y \psibf^i \;.
\end{split}
\end{equation}
Obviously, the tensors $\Aop = (A_{ijkl} )$, $\Cb = (C_{ij} )$ and
$\Kb = (K_{ij} )$ are symmetric, $\Aop$ adheres all the symmetries of
$\Dop$; moreover $\Aop$ is positive definite \chE{and $N > 0$}. The
hydraulic permeability $\Kb$ is positive semi-definite in general,
although it is positive definite whenever the channels intersect all
faces of $\pd Y$ and $Y_c$ is connected.

\subsection{Coupled flow deformation problem}

The Biot--Darcy model of poroelastic media for quasistatic problems
imposed in $\Om$ is constituted by the following equations involving
the homogenized coefficients:
\begin{equation}\label{eq-B1}
\begin{split}
-\nabla \cdot \sigmabf  & = \fb^s \;,\quad
\sigmabf  = \Aop \eeb{\ub} - \Bb p\;,\\
 -\nabla\cdot \wb & = \Bb:\eeb{\dot\ub} +  M \dot p  \;,\quad
\wb = -\frac{\Kb}{\bar\eta} \left (\nabla p - \fb^f \right)\;,\\
\end{split}
\end{equation}
where
\begin{equation}\label{eq-B1a}
\begin{split}
\Bb & := \Cb + \phi \Ib\;,\\
M & := N + \phi \gamma\;.
\end{split}
\end{equation}
Above $\gamma$ is the fluid compressibility and  $\phi = |Y_c|/|Y|$ is the volume fraction. The volume
effective forces in \eq{eq-B1} acting in the solid and fluid phases are denoted by
$\fb^s$ and $\fb^f$, respectively.

\subsection{Some poroelastic properties}
Poroelastic properties of porous media can be measured by experiments which consider an interplay between macroscopic quantities:
stress $\sigmabf$, strain $\eb$, fluid pressure $p$ and pore fluid increase $\zeta$, whereby
$\zeta = \nabla\cdot\wb$. Besides the relationship between $(\sigmabf,\zeta)$ and $(\eb,p)$, as defined in \eq{eq-B1}, the following equations establish the inverse mapping,
\begin{equation}\label{eq-B2}
\begin{split}
\eb & = \Cop \sigmabf + \Sb \zeta\;,\\
p & = -\Sb : \sigmabf + K \zeta\;,
\end{split}
\end{equation}
where the coefficients $\Cop$, $\Sb$ and $K$ can be expressed in terms of $\Aop$, $\Bb$ and $M$, as follows:
\begin{equation}\label{eq-B3}
\begin{split}
\mbox{ undrained compliance } \Cop & = K(\Iop - \Bb\otimes\Bb)\Aop^{-1}\;,\\
\mbox{ Skempton coefficients } \Sb & = K\Aop^{-1}\Bb\;,\\
\mbox{ a bulk modulus } K & = (M + \Bb\otimes\Bb : \Aop^{-1})^{-1}\;.
\end{split}
\end{equation}
Coefficient $\Cop$ can be measured during the undrained test, \ie $\zeta = 0$. Coefficients $\Sb$ and $K$ can be measured during the
jacketed test, by controlling amount of the pore fluid volume $\zeta$ while the specimen is unloaded, \ie $\sigmabf = 0$.

\begin{remark}\label{rem1}
To derive expressions for $\Cop $, $\Sb$, and $K$ in \eq{eq-B3}, we start with the mass conservation which now reads
$\Bb:\eb + M p = \zeta$, \ie, the pore fluid volume increase $\zeta$ leads to the pressure increase $p$ and to the macroscopic deformation $\eb$.
Using the constitutive law for the total stress $\sigmabf$, see \eq{eq-B1}, the straightforward manipulations yield
\begin{equation*}
\begin{split}
\eb & = \Aop^{-1}(\sigmabf + \Bb p)\;,\\
M p & = \zeta - \Bb: \Aop^{-1}(\sigmabf + \Bb p)\;,
\end{split}
\end{equation*}
hence
\begin{equation}\label{eq-B3a}
\begin{split}
p & = (\zeta - \Bb: \Aop^{-1}\sigmabf) / (M+ \Bb : \Aop^{-1}\Bb)\;,\\
\eb & = \big(\Aop^{-1}- (\Bb \otimes\Bb) \Aop^{-1} / (M+ \Bb : \Aop^{-1}\Bb)\big)\sigmabf \\
& \quad + \zeta \Aop^{-1}\Bb/ (M+ \Bb : \Aop^{-1}\Bb)\;,
\end{split}
\end{equation}
thus,  \eq{eq-B2} holds with the coefficients defined in \eq{eq-B3}.
Note the symmetry $\Bb:\Aop^{-1}\sigmabf = \sigmabf:\Aop^{-1}\Bb$.
\end{remark}

\section{Optimization problems}\label{sec-optprob}
We distinguish between the material optimization (MO) and the two-scale optimization (2SO).
In both MO and 2SO we use the shape optimization approach to provide globally, or locally desired microstructures.
In the latter case, the optimality of the material is achieved locally for a macroscopic problem with a given domain and specific boundary conditions.

\subsection{Material Optimization}\label{sec-matopt}
In this study we consider drained or undrained porous medium with periodic structure.
The criteria of optimal poroelastic material are related to the stiffness and permeability; the following problems are considered:

\begin{itemize}

\item \emph{Stiffness maximization} with a sufficient permeability guaranteed. The stiffness  is maximized \wrt preferred strain modes.
We require that the permeability $\Kb$ is sufficiently large compared to a given threshold, whereby anisotropy can be enforced according chosen preferential directions.
In the following text, these problems are referred to by abbreviations \textit{S/P}, \textit{S/P-bis} and \textit{S/PX};

\item \emph{Permeability maximization} with a sufficient stiffness guaranteed.
These optimization problems which are referred to by abbreviations \textit{P/S}, \textit{P/S-bis}, are reciprocal to the above formulations \textit{S/P}.

\item \emph{Compliance minimization of the undrained medium} with limited drained stiffness.
This problem, which is abbreviated by \textit{C/S},  involves all the Biot model coefficients, but the permeability is disregarded:
the objective is to minimize the undrained material compliance $\Cop$ for a loading such that there is no fluid redistribution.
\end{itemize}

Below we define problems \textit{S/P}, \textit{P/S} and \textit{C/S}, where the following functions are involved as the objective, or the constraint functions.
Let $\eb^k = (e_{ij}^k)$ be given strain modes and $\gb^k = (g_i^k)$  given directions, $k= 1,2,3$. We define:

\begin{equation}\label{eq-op1}
\begin{split}
\Phi_\eb^k(\Aop) = \eb^k: \Aop\eb^k \;, \quad \Psi^k(\Kb) = \gb^k\cdot\Kb \gb^k\;, \\
\Phi_\eb(\Aop) = \sum_k \gamma^k \Phi_\eb^k(\Aop)\;,
\quad \Psi(\Kb) = \sum_k \beta^k \Psi^k(\Kb)\;,
\end{split}
\end{equation}
where tensors $\eb^k=(e_{ij}^k)$ and vectors $\gb^k=(g_{i}^k)$ are linearly independent, i.e. $\sum_k \gamma^k \eb^k \not =\bmi{0}$ for $\sum_k |\gamma^k|\not = 0$,  and
$|\gb^1\cdot(\gb^2\times\gb^3)| > 0$.
The weights $\gamma^k$ and $\beta^k$ are to be given.

\subsubsection{Formulation --- problems S/P, S/P-bis and S/PX}\label{ssec-SP}
The aim is to maximize stiffness of the drained porous material and allow for a sufficient permeability.
These two properties characterizing the homogenized medium are related to tensors $\Aop$ and $\Kb$.
The state problem \eq{eq-h8} constitutes the constraint in the following design problem:
\begin{equation}\label{eq-op2}
\begin{split}
& \max_{\alphabf \in \Acal\times\RR} \Phi_\eb(\Aop) \\
\mbox{s.t.} \quad & \alphabf \rightarrow (\Aop,\Kb) \mbox{ by eqs. } \eq{eq-h8} \\
 & \kappa_0 \leq \Psi^k(\Kb)\;,\ k=1,2,3 \;, \\
& {\kappa_1 \leq \Psi(\Kb)}\;,\\
& r(|Y_c| - V_c^0 ) = 0\;,\\
 & |Y| = 1\;,
\end{split}
\end{equation}
where $r=1$ in the \textit{S/P-bis} problem, otherwise $r=0$, and by $V_c^0$ we denote the ``initial'' volume of $Y_c$.
By choosing $\kappa_0$ and $\kappa_1$ we distinguish between two alternative formulations of the optimization problem:
\begin{itemize}
\item \textit{S/P} or \textit{S/P-bis} problem: if $\kappa_0 > 0$ and $\kappa_1 \rightarrow -\infty$, then a sufficient permeability is guaranteed in all preferred directions.

\item \textit{S/PX} problem: if $\kappa_1 > 0$ and $\kappa_0 \rightarrow -\infty$, then only a weighted sum of the permeabilities is guaranteed.
\end{itemize}

We additionally apply periodic boundary conditions on the homologous control points \eq{eq-perbc} and injectivity constraints \eq{eq-inj}, as discussed in Section~\ref{sec-param}.


\subsubsection{Formulation --- problem P/S, P/S-bis and P/SX}\label{ssec-PS}
As a modification of the previously defined problem, now the role of the stiffness and permeability commutes.
The aim is to maximize permeability of the drained porous material whereby a sufficient stiffness must be achieved.
The design problem reads:
\begin{equation}\label{eq-op3}
\begin{split}
 & \max_{(\alphabf,\kappa) \in \Acal\times\RR} \Psi(\Kb)\\
\mbox{s.t.} \quad & \alphabf \rightarrow (\Aop,\Kb) \mbox{ by eqs. } \eq{eq-h8} \\
 & s_0 \leq \Phi_\eb^k(\Aop)\;,\ k=1,2,3 \;,\\
& s_1 \leq \Phi_\eb(\Aop)\;,\\
& r(|Y_c| - V_c^0 ) = 0\;,\\
 & |Y| = 1\;,
\end{split}
\end{equation}
where $r=1$ for the \textit{P/S-bis} problem, otherwise $r=0$.
Similar to the previous problems we distinguish between different formulations:
\begin{itemize}
\item \textit{P/S} or \textit{P/S-bis} problem: if $s_0 > 0$ and $s_1 \rightarrow -\infty$, then a sufficient stiffness is guaranteed in all preferred directions.
\item \textit{P/SX} problem: if $s_1 > 0$ and $s_0 \rightarrow -\infty$, then only a weighted sum of the directional stiffnesses is guaranteed.
\item \textit{P/SX'} problem: if $s_0 > 0$ and $s_1 > 0$, then stiffness in all preferred directions as well as a weighted sum of the stiffnesses is guaranteed.
\end{itemize}



\subsubsection{Formulation --- problem {C/S}}\label{ssec-SU}
This optimization problem merits minimization of the undrained material compliance $\Cop$, while the drained compliance $\Aop$ is kept above a given threshold.
Thus, all effective medium material parameters are involved in the formulation, see \eq{eq-B3}.
Let $\sigmabf = (\sigma_{ij})$ be a given stress mode and $\gb^k = (g_i^k)$ be given directions, as above. We now define
\begin{equation}\label{eq-op4}
\begin{split}
\Phi_\sigmabf(\Cop) = \sigmabf : \Cop \sigmabf\;.
\end{split}
\end{equation}
The optimal shape problem reads as follows,
\begin{equation}\label{eq-op5}
\begin{split}
 & \min_{\alphabf \in \Acal} \Phi_\sigmabf(\Cop) \\
\mbox{s.t.} \quad & \alphabf \rightarrow (\Aop,\Bb,M) \mbox{ by eqs. } \eq{eq-h8} \\
 & \Phi_\eb^k(\Aop) \leq s_0\;,\ k=1,2,3 \\
 & |Y| = 1\;.
\end{split}
\end{equation}
We recall that $\Cop$ is defined in  \eq{eq-B3}, thus $\Cop = K(\Iop - \Bb\otimes\Bb)\Aop^{-1}$ with
$K  = (M + \Bb\otimes\Bb : \Aop^{-1})^{-1}$.

\subsection{Towards 2-scale optimization}\label{sec-2scopt}
In contrast with the preceding section, where we considered the design of an optimal material, now we are looking for an optimal distribution of microstructures with slowly varying properties in a domain of interest with a specific loads and a desired flow regime.
Let $\Om \subset \RR^3$ be an open bounded domain. Its boundary splits, as follows:
$\pd\Om = \Gamma_D \cup \Gamma_N$ and also $\pd\Om = \Gamma_p \cup \Gamma_w$, where $\Gamma_D \cap \Gamma_N = \emptyset$ and $\Gamma_p \cap \Gamma_w = \emptyset$.

We consider the steady state problem for the linear Biot continuum
and give a formulation of the optimization problem of local microstructures $\Mcal(x)$, $x \in \Om$, for which the homogenized coefficients of the Biot continuum are computed.
By $\ub$ we denote the macroscopic displacement field, whereas by $P$ we refer to the macroscopic pressure.

Assuming the local microstructures $\Mcal(x)$ are given and the HC are computed,
we consider the macroscopic state problem: find $\ub$ and $P$ which satisfy
\begin{equation}\label{eq-opg1}
\begin{split}
-\nabla\cdot\left(\Aop \eeb{\ub} - P \Bb\right) & = 0\quad \mbox{ in } \Om\;,\\
\ub & = 0 \quad \mbox{ in }\Gamma_D\;,\\
\left(\Aop \eeb{\ub} - P \Bb\right) \cdot \nb & = \gb \quad \mbox{ in }\Gamma_N\;,
\end{split}
\end{equation}
and
\begin{equation}\label{eq-opg2}
\begin{split}
-\nabla\cdot\Kb\nabla P & = 0\quad \mbox{ in } \Om\;,\\
P & = \bar p^k \quad \mbox{ on } \Gamma_p^k\;,\quad k = 1,2\;,\\
\nb\cdot\Kb\nabla P & = 0 \quad \mbox{ on } \Gamma_w\;,\\
\end{split}
\end{equation}
where $\gb$ are the traction surface forces and $\bar p^k$ are given pressures on boundaries $\Gamma_p^k$, whereby $\Gamma_p$ consists of two disconnected, non-overlapping  parts,
$\Gamma_p = \Gamma_p^1\cup \Gamma_p^2$.
It is obvious that \eq{eq-opg1} splits into two parts: first, \eq{eq-opg2} can be solved for $P$, then \eq{eq-opg1} is solved for $\ub$.

Further, we consider an extension of $\bar p^k$ from boundary $\Gamma_p^k$ to the whole domain $\Om$, such that
$\bar p^k = 0$ on $\Gamma_p^l$ (in the sense of traces) for $l\not = k$ .
Then $P = p + \sum_k\bar p^k$ in $\Om$, such that $p = 0$ on $\Gamma_p$.
For the sake of the notation simplicity, we introduce $\bar p=\sum_k\bar p^k$.
By virtue of the Dirichlet boundary conditions for $\ub$ and $p$, we introduce the following spaces:
\begin{equation}\label{eq-opg3}
\begin{split}
V_0 & = \{\vb \in \Hdb(\Om)\,|\; \vb = 0 \mbox{ on } \Gamma_D\}\;,\\
Q_0 & = \{q \in H^1(\Om)\,|\; q = 0 \mbox{ on } \Gamma_p\}\;.
\end{split}
\end{equation}
We shall employ the bilinear forms and the linear functional $g$,
\begin{equation}\label{eq-opg4}
\begin{split}
\aOm{\ub}{\vb} & = \int_\Om (\Aop\eeb{\ub}):\eeb{\vb}\;,\\
\bOm{p}{\vb} & = \int_\Om p \Bb:\eeb{\vb}\;,\\
\cOm{p}{q} & = \int_\Om \nabla q\cdot \Kb\nabla p\;,\\
g(\vb) & = \int_{\Gamma_N} \gb \cdot \vb\;.
\end{split}
\end{equation}
The weak formulation reads, as follows: find $\ub\in V_0$ and $p \in Q_0$, such that
\begin{equation}\label{eq-opg5}
\begin{split}
\aOm{\ub}{\vb} - \bOm{p}{\vb} & = g(\vb) + \bOm{\bar p}{\vb}\;,\quad \forall \vb \in V_0\;,\\
\cOm{p}{q} & = -\cOm{\bar p}{q}\;,\quad q \in Q_0\;.
\end{split}
\end{equation}
Since the two fields are decoupled, first $p$ is solved from \eq{eq-opg5}$_2$, then
$\ub$ is solved from \eq{eq-opg5}$_1$, where $p$ is already known.

By $\alpha(x)$ we denote an abstract optimization variable which determines the homogenized coefficients for any position $x\in\Om$.
The objective function $\Phi_\alpha(\ub)$ and the constraint function $\Psi_\alpha(p)$ are defined, as follows:
\begin{equation}\label{eq-opg6}
\begin{split}
\Phi_\alpha(\ub) & = g(\ub)\;,\\
\Psi_\alpha(p) & = - \int_{\Gamma_p^2}\Kb\nabla (p+ \bar p) \cdot \nb\;.
\end{split}
\end{equation}
The constraint function $\Psi_\alpha(p)$ expresses the fluid flow  through surface $\Gamma_p^2$ due to the pressure difference $\bar p^1 - \bar p^2$, see
the boundary condition  \eq{eq-opg2}$_2$.

Let $\Acal$ be a set of admissible designs $\alpha(x) \mapsto \Mcal(x)$.
Now the macroscopic abstract optimization problem reads:
\begin{equation}\label{eq-opg7}
\begin{split}
\min_{\alpha \in \Acal}\ & \Phi_\alpha(\ub)\\
\mbox{ s.t. } & \Psi_\alpha(p) = \ol{\Psi_0}\;,\\
 & (\ub,p) \mbox{ satisfies } \eq{eq-opg5}\;.
\end{split}
\end{equation}
This optimization problem is associated with the following inf-sup problem, 
\begin{equation}\label{eq-opg9}
\begin{split}
\min_{\alpha \in \Acal} \inf_{(\ub,p) \in V_0\times Q_0} \sup_{ \Lambda \in \RR, (\tilde\vb,\tilde q)\in V_0\times Q_0} & \Lcal(\alpha, (\ub,p),\Lambda, (\tilde\vb,\tilde q))\;,
 \end{split}
\end{equation}
with the Lagrangian function,
\begin{equation}\label{eq-opg8}
\begin{split}
\Lcal(\alpha, (\ub,p),\Lambda, (\tilde\vb,\tilde q)) & =
\Phi_\alpha(\ub) + \Lambda (\Psi_\alpha(p)-\ol{\Psi_0}) + \aOm{\ub}{\tilde\vb} - \bOm{p+\bar p}{\tilde\vb} \\
& \quad - g(\tilde\vb) + \cOm{p+\bar p}{\tilde q}\;,
\end{split}
\end{equation}
where $\Lambda\in \RR$ and $(\tilde\vb,\tilde q)\in V_0\times Q_0$ are the Lagrange multipliers associated with the constraints of the problem \eq{eq-opg7}.

For a while, let us consider the effective parameters
$\Hop = (\Aop,\Bb,\Kb)$ as the optimization variables which can be parameterized by $\alpha \in \Acal$, and assume
a given value $\Lambda \in \RR$; note that $\Lambda$ can be positive or negative depending on the desired flow augmentation, or reduction. Now the optimization
problem \eq{eq-opg7} can be rephrased as the two-criteria minimization problem,
\begin{equation}\label{eq-opH1}
\begin{split}
 \min_{
\begin{array}{c}
\Hop(\alpha) \\
\alpha \in \Acal
\end{array}
} \ &  \Fcal(\Hop(\alpha),\zb)\;,\quad\Fcal(\Hop,\zb) = \Phi(\Hop,\ub) + \Lambda\Psi(\Hop,p)\;,\\
\mbox{ s.t. } & \zb = (\ub,p) \mbox{ satisfies } \eq{eq-opg5}\;.
\end{split}
\end{equation}
For any admissible state $\zb = (\ub,p)$ such that $\Hop \mapsto \zb(\Hop)$, it holds that
$\Lcal(\Hop,\zb,\Lambda,\tilde\zb) = \Fcal(\Hop,\zb)$; in the notation employed here we use the primary dependence of the bilinear forms \eq{eq-opg4} on $\Hop$.

To solve \eq{eq-opg7}, we can proceed by linearization of $\Fcal(\Hop,\zb(\Hop))$ with respect to $\Hop$, whereby the state variables $\zb$ are related to $\Hop$ by the state problem constraint.
Therefore, we introduce the total differential $\dlt_\Hop^\tot$, such that $\dlt_\Hop^\tot \Fcal  \circ \dlt\Hop = (\dlt_\Hop \Fcal + \dlt_\zb \Fcal\circ \dlt_\Hop \zb)\circ\dlt\Hop$ with
$\dlt_\Hop \zb$ expressing the sensitivity of the state \wrt the effective parameters.

Further we shall label quantities related to a reference parameter $\Hop_0$ by subscript $_0$. The linearization of $\Fcal(\Hop,\zb(\Hop))$ at the reference
layout represented by $\Hop_0$ yields the following approximation $\wtilde{\Fcal}$ of the objective function,
\begin{equation}\label{eq-opH2}
\begin{split}
\wtilde{\Fcal}(\Hop)  =  \wtilde{\Lcal}(\Hop) & = \Lcal_0 + \dlt_\Hop^\tot \Lcal_0 \circ(\Hop-\Hop_0)\\
& = \Fcal(\Hop_0)  + \dlt_\Hop^\tot\Fcal(\Hop_0)\;,
\end{split}
\end{equation}
where $\Lcal_0 := \Lcal(\Hop_0,\zb_0,\Lambda,\tilde\zb_0) = \Phi(\Hop_0,\ub_0) + \Lambda\Psi(\Hop_0,p_0)$ because of the state $\zb_0$ consistency with $\Hop_0$, \ie
$\Hop_0\mapsto \zb_0$ by virtue of \eq{eq-opg5}. The sensitivity $\dlt_\Hop^\tot \Lcal_0$ is derived below using the adjoint state method; we obtain
\begin{equation}\label{eq-opH3}
\begin{split}
 \dlt_\Hop^\tot \Lcal_0 & = \dlt_\Hop \Phi(\ub_0) + \Lambda  \dlt_\Hop \Psi(p_0)\\
& \quad \quad + \dlt_\Hop\aOm{\ub_0}{\tilde\vb_0}-\dlt_\Hop\bOm{p_0+\bar p}{\tilde\vb_0}+\dlt_\Hop\cOm{p_0+\bar p}{\tilde q_0}\;,
\end{split}
\end{equation}
where $\tilde\zb_0 = (\tilde\vb_0,\tilde q_0)$ is the adjoint state.
Hence, using the definitions \eq{eq-opg4}, we obtain (note that $P_0 = p_0 + \bar p$),
\begin{equation}\label{eq-opH4}
\begin{split}
 \dlt_\Hop^\tot \Lcal_0 \circ(\Hop-\Hop_0) & =
\int_\Om (\Aop-\Aop_0)\eeb{\ub_0}:\eeb{\tilde\vb_0} - \int_\Om P_0 (\Bb - \Bb_0):\eeb{\tilde\vb_0}\\
& \quad + \int_\Om \nabla \tilde q_0 \cdot (\Kb-\Kb_0)\nabla P_0 - \Lambda\int_\Om \nabla \tilde p \cdot (\Kb-\Kb_0)\nabla P_0
\;,\\
 \dlt_\Hop^\tot \Lcal_0 \circ \Hop_0 & = \Lcal_0  = \Phi(\ub_0) + \Lambda \Psi(P_0)\;,
\end{split}
\end{equation}
where $\tilde p$ is a given function independent of $\Hop$, see \eq{eq-opg12}.
The expression \eq{eq-opH4}$_1$ is the straightforward consequence of \eq{eq-opH3} and \eq{eq-opg13a}, whereas
\eq{eq-opH4}$_2$ holds due to the consistency $\zb_0=\zb(\Hop_0)$ and linearity of $\Lcal$ \wrt $\Hop$.
Since $\dlt_\Hop^\tot \Fcal(\Hop_0) = \dlt_\Hop^\tot \Lcal_0$,
from \eq{eq-opH2} and using \eq{eq-opH4} we obtain
\begin{equation}\label{eq-opH5}
\begin{split}
\wtilde{\Fcal}(\Hop) & = \Lambda \ol{\Psi_0} + \dlt_\Hop^\tot \Fcal(\Hop_0)\circ \Hop \\
& =  \Lambda \ol{\Psi_0} + \int_\Om \Aop\eeb{\ub_0}:\eeb{\tilde\vb_0} - \int_\Om P_0 \Bb :\eeb{\tilde\vb_0}\\
& \quad + \int_\Om \nabla \tilde q_0 \cdot \Kb \nabla P_0 - \Lambda\int_\Om \nabla \tilde p \cdot \Kb\nabla P_0\;.
\end{split}
\end{equation}

The homogenized effective medium parameters $\Hop$ can be defined pointwise in $\Om$, assuming their moderate variation \wrt the position.
However, by virtue of the local periodicity assumption and also for practical reasons related to the numerical solutions,
we consider a decomposition of $\Om$ into a finite number of non-overlapping subdomains $\Om_e$ in which the design of microstructures and, thereby, the effective medium parameters $\Hop$ are constant, \ie $\alpha(x) = \alpha_e$ for $x \in \Om_e$.
We recall that parameters $\Hop(x) = \Hop_e = \Hop(\alpha_e)$ are defined by implicit functions of $\alpha_e$ by virtue of the homogenization problem.

From the linearized objective functional $\wtilde{\Fcal}(\Hop)$, see \eq{eq-opH5}, the following local nonlinear optimization problems associated with the macroscopic element $\Om_e$ can be deduced:
\begin{equation}\label{eq-opg16}
\begin{split}
& \min_{\alpha \in \Acal} F_e(\alpha)\\
\mbox{ where }\quad  F_e & := \int_{\Om_e}
\left[\Aop(\alpha) :: (\eeb{\ub}\otimes\eeb{\tilde\vb})- P \Bb(\alpha):\eeb{\tilde\vb}\right. \\
& \quad \left.- \Lambda\Kb(\alpha):(\nabla P \otimes\nabla\tilde p) + \Kb(\alpha):(\nabla P \otimes\nabla\tilde q)\right]
 \;,
\end{split}
\end{equation}
where $\tilde p$ is defined in the context of \eq{eq-opg12}; we recall $P = p+\bar p$ is the total pressure inolved in the original state problem formulation \eq{eq-opg1}.
Obviously, the state variables $(\ub,p)$ as well as the adjoint state   variables $(\tilde \vb,\tilde q)$ are given by the macroscopic optimization step.
Thus, the local microstructures can be optimized
for the computed macroscopic response which provides the tensors
\begin{equation}\label{eq-opg17}
\begin{split}
\eeb{\ub}\otimes\eeb{\tilde\vb}, \quad \nabla P \otimes\nabla\tilde q \;,\quad  \nabla P \otimes\nabla\tilde p \quad \mbox{ and } \quad   P \eeb{\tilde\vb}\;.
\end{split}
\end{equation}


\subsection{Adjoint responses and the sensitivity analysis}
In this section we provide details concerning the sensitivity analysis employed in the preceding section.
We consider $\alpha$ to represent a general optimization variable which is related to the effective medium parameters $\Hop$.
It is worth to note that one may also consider $\alpha \equiv \Hop$ in the context of the free material optimization (FMO).

To obtain the adjoint equation,  we consider  the optimality condition for $(\ub,p)$, thus, from \eq{eq-opg9} it follows that
\begin{equation}\label{eq-opg10}
\begin{split}
\dlt_{(\ub,p)}\Lcal(\alpha, (\ub,p),\Lambda, (\tilde\vb,\tilde q))\circ (\vb,q) & =
\dlt_\ub\Phi_\alpha(\ub;\vb) + \Lambda \dlt_p\Psi_\alpha(p; q)  \\
& \quad + \aOm{\vb}{\tilde\vb} - \bOm{q}{\tilde\vb} + \cOm{q}{\tilde q} \;,
\end{split}
\end{equation}
where
\begin{equation}\label{eq-opg11a}
\begin{split}
\dlt_\ub\Phi_\alpha(\ub;\vb) = g(\vb)\;,\quad\mbox{ and }\quad \dlt_p\Psi_\alpha(p; q) = -\int_{\Gamma_p^2} \Kb\nabla q \cdot\nb\;.
\end{split}
\end{equation}
To avoid computing of the gradient $\nabla q$ on $\Gamma_p^2\subset \pd \Om$, we consider $\tilde p \in H^1(\Om)$ such that
$\tilde p = 0$ on $\Gamma\setminus\Gamma_p^2$, while $\tilde p = 1$ on $\Gamma_p^2$, then it is easy to see that
\begin{equation}\label{eq-opg12}
\begin{split}
-\Psi_\alpha(p) & = r(p) := \cOm{p+\bar p}{\tilde p}\;,\\
-\dlt_p\Psi_\alpha(p; q) & = \dlt_p r(p;q) = \cOm{q}{\tilde p}\;.
\end{split}
\end{equation}
The optimality conditions \eq{eq-opg10} related to the state admissibility  yield
the adjoint state $(\tilde\vb,\tilde q)$ which satisfies the following identities:
\begin{equation}\label{eq-opg11b}
\begin{split}
\tilde\vb \in V_0\;:\quad \aOm{\vb}{\tilde\vb} & = - \dlt_\ub\Phi_\alpha(\ub;\vb)\quad \forall \vb \in V_0\;,\\
\tilde q \in Q_0\;:\quad \cOm{q}{\tilde q} & = \bOm{q}{\tilde\vb} - \Lambda\dlt_p\Psi_\alpha(p; q),\quad \forall q \in Q_0\;.
\end{split}
\end{equation}
These equations can be rewritten using \eq{eq-opg11a} and \eq{eq-opg12}, as follows
\begin{equation}\label{eq-opg11}
\begin{split}
\tilde\vb \in V_0\;:\quad \aOm{\vb}{\tilde\vb} & = -  g(\vb)\quad \forall \vb \in V_0\;,\\
\tilde q \in Q_0\;:\quad \cOm{q}{\tilde q} & = \bOm{q}{\tilde\vb} + \Lambda \cOm{q}{\tilde p},\quad \forall q \in Q_0\;.
\end{split}
\end{equation}
We can compute the total variation of the Lagrangian,
\begin{equation}\label{eq-opg13}
\begin{split}
\dlt_\alpha^\tot \Lcal & = \dlt_\ub g(\ub;\dlt_\alpha\ub) - \Lambda \dlt_p r(p;\dlt_\alpha p) + \dlt_\alpha g(\ub) - \Lambda  \dlt_\alpha r(p)
\\
& \quad\quad + \aOm{\dlt_\alpha\ub}{\tilde\vb} - \bOm{\dlt_\alpha p}{\tilde\vb}
+\cOm{\dlt_\alpha p}{\tilde q}\\
& \quad \quad + \dlt_\alpha\aOm{\ub}{\tilde\vb}-\dlt_\alpha\bOm{p+\bar p}{\tilde\vb}
+\dlt_\alpha\cOm{p+\bar p}{\tilde q}\;.
\end{split}
\end{equation}
If the pair $(\ub,p)$  solves the state problem and $(\tilde\vb,\tilde q)$ is its adjoint state,
\eq{eq-opg13} is equivalent to the following expression,
\begin{equation}\label{eq-opg13a}
\begin{split}
\dlt_\alpha^\tot \Lcal & = \dlt_\alpha g(\ub) - \Lambda  \dlt_\alpha r(p)\\
& \quad \quad + \dlt_\alpha\aOm{\ub}{\tilde\vb}-
\dlt_\alpha\bOm{p+\bar p}{\tilde\vb}
+\dlt_\alpha\cOm{p+\bar p}{\tilde q}\;.
\end{split}
\end{equation}
Above the shape derivatives $\dlt_\alpha$ of the bilinear forms can be rewritten in terms of the sensitivity of the homogenized coefficients.
Besides the obviously vanishing derivative $\dlt_\alpha g(\ub)=0$, it holds that
\begin{equation}\label{eq-opg15}
\begin{split}
\dlt_\alpha\aOm{\ub}{\tilde\vb}\circ\dlt_\alpha\Aop & =
\sum_e \int_{\Om_e} \dlt_\alpha\Aop\eeb{\ub}:\eeb{\tilde\vb}\;,\\
\dlt_\alpha\bOm{p+\bar p}{\tilde\vb}\circ\dlt_\alpha\Bb & =
\sum_e \int_{\Om_e} (p+\bar p) \dlt_\alpha\Bb:\eeb{\tilde\vb}\;,\\
\dlt_\alpha\cOm{p+\bar p}{\tilde q}\circ\dlt_\alpha\Kb & =
\sum_e \int_{\Om_e} \nabla \tilde q \cdot \dlt_\alpha\Kb\nabla( p+\bar p)\;,\\
\dlt_\alpha \cOm{p+\bar p}{\tilde p}\circ\dlt_\alpha\Kb & =
\sum_e \int_{\Om_e} \nabla \tilde p \cdot \dlt_\alpha\Kb\nabla( p+\bar p)\;.
\end{split}
\end{equation}

\section{Shape sensitivity of the homogenized coefficients}\label{sec-sa}

We use the shape sensitivity technique and the material derivative
approach (see
e.g. \cite{haslinger_makinen_2003:_introduction_shape_optimization,haug_choi_komkov_1986:_design_sensitivity_analysis})
to obtain the sensitivity of homogenized coefficients \eq{eq-h8}
involved in the Biot continuum \eq{eq-B1} with respect to the
configuration transformation corresponding to the  Y-periodic vector field $\vec\Vcal(y)$, $y \in Y$
so that for $y \in \Gamma$ it describes the convection of points on the
interface.  Although we pursue the general approach reported e.g. in
\cite{Rohan2003,Rohan-AMC}, below we shall discuss further details.
Using $\vec\Vcal: \overline{Y} \longrightarrow \RR^3$ defined
in $Y$, we can parameterized the ``material point'' position in
$Y$ by $z_i(y,\tau) = y_i + \tau\Vcal_i(y)$, $y \in Y$, $i=1,2,3$, where
$\tau$ is the ``time-like'' variable. Throughout the text below we
shall use the notion of the following derivatives: $\delta(\cdot)$ is
the total ({\it material}) derivative, $\delta_\tau(\cdot)$ is the
partial ({\it local}) derivative {\it w.r.t.} $\tau$.  These
derivatives are computed as the directional derivatives in the
direction of $\vec\Vcal(y)$, $y \in Y$.

Let us consider a general functional $\Phi(\phi) = \int_{Y_d}
F(\phi)$, where $\phi(y)$ corresponds to a microscopic state function,
$Y_d \subset Y$, and $F$ is a sufficiently regular operator.  Using
the chain rule differentiation, the total sensitivity of $\Phi$ is
given as
\begin{equation}\label{eq-S6}
\begin{split}
\delta \Phi(\phi) & = \delta_\phi \Phi(\phi)\circ \delta \phi + \delta_\tau \Phi(\phi) = \delta_\phi \Phi(\phi)\circ \delta \phi +
\int_{Y_d} F(\phi) \dvg\vec\Vcal + \int_{Y_d} \delta_\tau F(\phi)\;,
\end{split}
\end{equation}
where $ \delta_\tau F(\phi)$ is the shape derivative of $F(\phi)$ for
a fixed argument $\phi$. It is desirable to eliminate the sensitivity
$\delta \phi$ from the sensitivity formula; for this the microscopic
state problem are used without need for solving any adjoint problem.

\subsection{Shape sensitivity analysis of the permeability}\label{sec-defsaK}

We are interested in the influence of variation of the shape of the
interface $\Gamma$ on the homogenized permeability $K_{ij}$ defined in
\eq{eq-h8}.  Thus, according to \eq{eq-S6} with $F \equiv I$, the
differentiation of $K_{ij}$ yields
\begin{equation}\label{eq-S7}
\begin{split}
\delta K_{ij} = \delta_\tau\intY_{Y_c}\psi_j^i + \intY_{Y_c}\delta\psi_j^i\;.
\end{split}
\end{equation}
To eliminate the dependence on $\delta\psi_j^i$ in the last integral
of the above expression, we differentiate \eq{eq-S3}$_1$ which yields
\begin{equation}\label{eq-S8}
\begin{split}
\intY_{Y_c}\nabla_y\delta\psibf^i:\nabla_y \vb -
\intY_{Y_c}\delta\pi^i \nabla\cdot\vb + \delta_\tau\left(
\intY_{Y_c}\nabla_y\psibf^i:\nabla_y \vb -
\intY_{Y_c}\pi^i \nabla\cdot\vb\right) = \delta_\tau\intY_{Y_c}v_i\;.
\end{split}
\end{equation}
We proceed by substitution $\vb = \psibf^j$ in \eq{eq-S8} and use the
incompressibility constraint \eq{eq-S3}$_2$. Then, from \eq{eq-S3}$_1$
evaluated for $k= j$ and substituting there $\vb = \delta \psibf^i$ we
get
\begin{equation}\label{eq-S9}
\begin{split}
\intY_{Y_c} \delta\psi_j^i & = \intY_{Y_c}\nabla_y\delta\psibf^i:\nabla_y\psibf^j
- \intY_{Y_c} \pi^j \nabla_y \cdot \delta\psibf^i \\
& = \delta_\tau \left(
\intY_{Y_c} \psi_i^j - \intY_{Y_c} \nabla_y\psibf^i:\nabla_y\psibf^j
+  \intY_{Y_c} \pi^i \nabla_y \cdot \psibf^j
\right) + \delta_\tau \intY_{Y_c} \pi^j \nabla_y \cdot \psibf^i
\;,
\end{split}
\end{equation}
where we employed differentiated \eq{eq-S3}$_2$ with $q = \pi^j$
afterwards and used \eq{eq-S8}.  Now we can substitute into \eq{eq-S7}
which yields
\begin{equation}\label{eq-S10}
\begin{split}
\delta K_{ij} & = \delta_\tau  \intY_{Y_c} \left(  \psi_i^j + \psi_j^i
- \nabla_y\psibf^i:\nabla_y\psibf^j\right) + \delta_\tau  \intY_{Y_c}\left(
\pi^i \nabla_y \cdot \psibf^j +\pi^j \nabla_y \cdot \psibf^i\right)\;.
\end{split}
\end{equation}
To compute the shape derivative $\delta_\tau$, we shall employ the following expressions:
\begin{equation}\label{eq-S11a}
\begin{split}
\delta_\tau \intY_{Y_c}\psi_i^j & =
\intY_{Y_c}\psi_i^j \left(\nabla_y\cdot\vec\Vcal- \intY_Y\nabla_y\cdot \vec\Vcal\right)\;,\\
\delta_\tau \intY_{Y_c}\nabla_y\psibf^j:\nabla_y\psibf^i & =
\intY_{Y_c}\left(
\nabla_y\psibf^j:\nabla_y\psibf^i \nabla_y\cdot\vec\Vcal-
\pd_l^y\Vcal_r \pd_r^y\psi_k^i\pd_l^y\psi_k^j - \pd_l^y\Vcal_r \pd_r^y\psi_k^j\pd_l^y\psi_k^i\right) \\
& \quad - \intY_{Y_c}\nabla_y\psibf^j:\nabla_y\psibf^i \intY_Y\nabla_y\cdot \vec\Vcal\;,\\
\delta_\tau \intY_{Y_c} \nabla_y\cdot \psibf^j
& = \intY_{Y_c} \nabla_y\cdot \psibf^j(\nabla_y \cdot \vec\Vcal - \intY_Y\nabla_y\cdot \vec\Vcal)- \intY_{Y_c} \pd_k^y\Vcal_r \pd_r^y\psi_k^j\;.
\end{split}
\end{equation}

Finally, using \eq{eq-S10} and \eq{eq-S11a} we get the following expression
\begin{equation}\label{eq-S11}
\begin{split}
\delta K_{ij} & = \intY_{Y_c}\left( \psi_i^j + \psi_j^i- \nabla_y\psibf^i:\nabla_y\psibf^j
+\pi^i \nabla_y \cdot \psibf^j +\pi^j \nabla_y \cdot \psibf^i\right)\left(\nabla_y\cdot \vec\Vcal - \intY_Y\nabla_y\cdot \vec\Vcal\right)\\
& \quad + \intY_{Y_c} \left(
\pd_l^y\Vcal_r \pd_r^y\psi_k^i\pd_l^y\psi_k^j + \pd_l^y\Vcal_r \pd_r^y\psi_k^j\pd_l^y\psi_k^i -
\pi^i \pd_k^y\Vcal_r \pd_r^y\psi_k^j
-  \pi^j \pd_k^y\Vcal_r \pd_r^y\psi_k^i
\right)\;.
\end{split}
\end{equation}
Although $\delta K_{ij}$ depends formally on $\vec\Vcal$, this field can
be constructed arbitrarily in the interior of $Y_c$ without any
influence on $\delta K_{ij}$.

\subsection{Deformation sensitivity analysis of the poroelasticity coefficients}
\label{sec-defsaPE}

In contrast with the sensitivity of the permeability coefficients
which depend on the shape of $\pd Y_c$ only, the poroelasticity may
depend on the strain associated with $\vec\Vcal(y)$, $y \in Y_m$.

By virtue of the shape sensitivity based on the domain parametrization
\cite{haslinger_makinen_2003:_introduction_shape_optimization,haug_choi_komkov_1986:_design_sensitivity_analysis}, the
following formulae hold,
\begin{equation}\label{eq-S15}
\begin{split}
\delta_\tau \aYm{\ub}{\vb} &= \intY_{Y_m}
D_{irks}\left(\delta_{rj}\delta_{sl} \nabla_y\cdot \vec\Vcal - \delta_{jr}\pd_s^y \Vcal_l - \delta_{ls}\pd_r^y \Vcal_j\right)
e_{kl}^y(\ub) e_{ij}^y(\vb)\\
& \quad - \aYm{\ub}{\vb}\intY_{Y} \nabla_y\cdot \vec\Vcal\;,\\
\end{split}
\end{equation}
\begin{equation}\label{eq-S16}
\begin{split}
\delta_\tau \intY_{Y_m} \nabla_y\cdot\vb & = \intY_{Y_m}\left(
\nabla_y \cdot \vec\Vcal \nabla_y \cdot \vb - \pd_i^y \Vcal_k \pd_k^y \vb_i
\right) - \intY_{Y_m} \nabla_y\cdot\vb\intY_{Y} \nabla_y\cdot \vec\Vcal\;.
\end{split}
\end{equation}
To differentiate \eq{eq-h5a}, we also employ $\delta_\tau \Pi_k^{ij} = \Vcal_j \delta_{ik}$.  We shall need the
sensitivity identity obtained by differentiation in \eq{eq-h5b} which
yields
\begin{equation}\label{eq-S17}
\begin{split}
\aYm{\delta \omegabf^P}{\vb} & = \delta_\tau \left(\intY_{Y_m} \nabla_y\cdot \vb -
\aYm{\omegabf^P}{\vb}\right)\;.
\end{split}
\end{equation}
We can now differentiate the expressions for the homogenized
coefficients given in \eq{eq-h8}. First we obtain the sensitivity of
$A_{ijkl}$, whereby \eq{eq-h5a} is employed:
\begin{equation}\label{eq-S18}
\begin{split}
\delta A_{ijkl} = \delta_\tau
\aYm{\omegabf^{ij}+\Pibf^{ij}}{\omegabf^{kl}+\Pibf^{kl}}
+ \aYm{\delta_\tau\Pibf^{ij}}{\omegabf^{kl}+\Pibf^{kl}}
+\aYm{\omegabf^{ij}+\Pibf^{ij}}{\delta_\tau \Pibf^{kl}}\;.
\end{split}
\end{equation}
The sensitivity of $C_{ij}$ is obtained, as follows:
\begin{equation}\label{eq-S19}
\begin{split}
\delta C_{ij} & = \aYm{\delta \omegabf^P}{\Pibf^{ij}}
+\delta_\tau\aYm{\omegabf^P}{\Pibf^{ij}} +\aYm{\omegabf^P}{\delta_\tau\Pibf^{ij}}\\
& = \delta_\tau\aYm{\omegabf^P}{\omegabf^{ij}} -
\delta_\tau\intY_{Y_m}\nabla_y\cdot\omegabf^{ij} +\delta_\tau\aYm{\omegabf^P}{\Pibf^{ij}} +\aYm{\omegabf^P}{\delta_\tau\Pibf^{ij}}\;,
\end{split}
\end{equation}
where we used \eq{eq-S17} with $\vb$ substituted by $\omegabf^{ij}$
and the first \rhs integral in \eq{eq-S19}$_1$ was rewritten using
\eq{eq-h5a}.

The sensitivity of $N$ is derived using \eq{eq-S17} with $\vb$
substituted by $\omegabf^P$; thus, we obtain
\begin{equation}\label{eq-S20}
\begin{split}
\delta N & = 2\aYm{\delta\omegabf^P}{\omegabf^P} + \delta_\tau \aYm{\omegabf^P}{\omegabf^P} = 2 \delta_\tau\intY_{Y_m}\nabla_y\cdot\omegabf^P - \delta_\tau \aYm{\omegabf^P}{\omegabf^P}\;.
\end{split}
\end{equation}
Finally, the sensitivity of the volume fraction is 
\begin{equation}\label{eq-S20a}
\begin{split}
\delta \phi = \intY_{Y_c} \nabla_y \cdot \vec\Vcal - \phi \intY_{Y} \nabla_y \cdot \vec\Vcal\;.
\end{split}
\end{equation}

To summarize, the sensitivity formulae \eq{eq-S18}-\eq{eq-S20} can be
evaluated using expressions \eq{eq-S15} and \eq{eq-S16}. Using
\eq{eq-S19}, \eq{eq-S20} and \eq{eq-S20a} we obtain the sensitivities
of $\Bb$ and $M$, i.e.
\begin{equation}\label{eq-S20b}
\begin{split}
\delta \Bb = \delta \Cb + \delta \phi \Ib\;, \quad \delta M =
\delta N + \gamma \delta \phi\;.
\end{split}
\end{equation}

\begin{remark}\label{rem2}
By virtue of \eq{eq-B3}, the sensitivities of $\Cop$, $\Sb$ and $K$ can be expressed in terms of $\delta \Aop$, $\delta \Bb$, and $\delta M$.
The following relationships are straightforward:
\begin{equation}\label{eq-S21b}
\begin{split}
\dlt \Cop & = \dlt K (\Iop - \Bb\otimes\Bb)\Aop^{-1} - K(\dlt\Bb\otimes\Bb + \Bb\otimes\dlt\Bb)\Aop^{-1} \\
& \quad -  K (\Iop - \Bb\otimes\Bb)\Aop^{-1}(\dlt \Aop) \Aop^{-1}\;,
\end{split}
\end{equation}
where
\begin{equation}\label{eq-S22b}
\begin{split}
\dlt K & = -K^2 \dlt(1/K)\;,\\
\dlt(1/K) & = \dlt M + (\dlt\Bb\otimes\Bb + \Bb\otimes\dlt\Bb):\Aop^{-1} -  \Bb\otimes\Bb :\Aop^{-1}(\dlt \Aop) \Aop^{-1}\;.
\end{split}
\end{equation}
We used $\dlt(\Aop^{-1}) = -\Aop^{-1}(\dlt \Aop) \Aop^{-1}$.
\end{remark}

\section{Design parametrization}\label{sec-param}

The design is parameterized by control points $\{\Pb^\imx\}$ of a spline box with a multi-index $\imx = (i_1,i_2,i_3)$.
The box is expressed in terms of spline basis functions
\begin{equation}
\Bcal^\imx : \RR^3 \rightarrow \RR\;,\ \Bcal^\imx(\tb) = \Bcal^{i_1} (t_1)\Bcal^{i_2} (t_2)\Bcal^{i_3} (t_3)
\end{equation}
with $\tb = (t_1,t_2,t_3) \in [0,1]^3$ and hence is a tensor product volume.
The functions $\Bcal^{i_k}$ are B-spline basis functions, which are defined with respect to clamped knot vectors.
Thus, we define
\begin{equation}\label{eq-spl1}
\begin{split}
\xb(\tb,\alphabf) = \sum_\imx \Pb^{\,\imx}(\alphabf^\imx) \Bcal^\imx(\tb)\;,\ \Pb^{\,\imx}(\alphabf^\imx) = \hat\Pb^{\,\imx} + \alphabf^\imx\;,
\end{split}
\end{equation}
where $\hat\Pb^{\,\imx}$ are the initial control points and $\alphabf^\imx$ are parameter which serve as design variables in all the optimization problems described in sections \ref{sec-optprob}
and modify the control points of the spline.
For better readability we will omit the argument $\alphabf^\imx$ of $\Pb^\imx$ below.

A reference cubic spline box is defined by a lattice of the control points $\hat\Pb^{\,\imx}$ with $i_j \in \{1,\ldots,m_j\}$, where $m_j$ is given by the spline order and the number of segments in the $j$-th direction (see figure \ref{fig-spb}).
The initial control points are placed at the Greville abscissae possibly resulting in a nonregular lattice.

Lattice periodicity is considered for all control points at the boundary of the spline box.
We call an index $\imx = (i_1,i_2,i_3)$ with $i_k = 1$ for at least one $k \in \{1,2,3\}$ a master index labeled by $\imx^M$
and an index with $i_k = m_k$ for at least one $k \in \{1,2,3\}$ a slave index labeled by $\imx^S$.
Simply choosing $\alphabf^{\imx^S}=\alphabf^{\imx^M}$ would only permit for deformations where the eight outermost vertices still define a cube.
To account for more flexible designs we add additional design variables $\betabf_k\in\RR^3,k=1,2,3,$ and define
\begin{equation}
\betabf^\imx_k =
\begin{cases}
\betabf_k &\text{ if } i_k = m_k\;, \\
0 &\text{ else}\;,
\end{cases}
\end{equation}
and
\begin{equation}
\alphabf^\imx = \tilde\alphabf^\imx + \sum_{k=1}^3 \betabf^\imx_k\;.
\end{equation}
Thus, each $\betabf_k$ translates one (slave) face of the lattice without deforming it allowing for sheared boxes as well.
The lattice periodicity is then established by
\begin{equation}\label{eq-perbc}
\Pb^{\imx^S} = \Pcal_\# \Pb^{\imx^M}\;,
\end{equation}
where $\Pcal_\#$ ensures $\tilde\alphabf^{\imx^S}=\tilde\alphabf^{\imx^M}$.

\begin{figure}[tp]
\begin{center}
  \includegraphics[width=.79\textwidth]{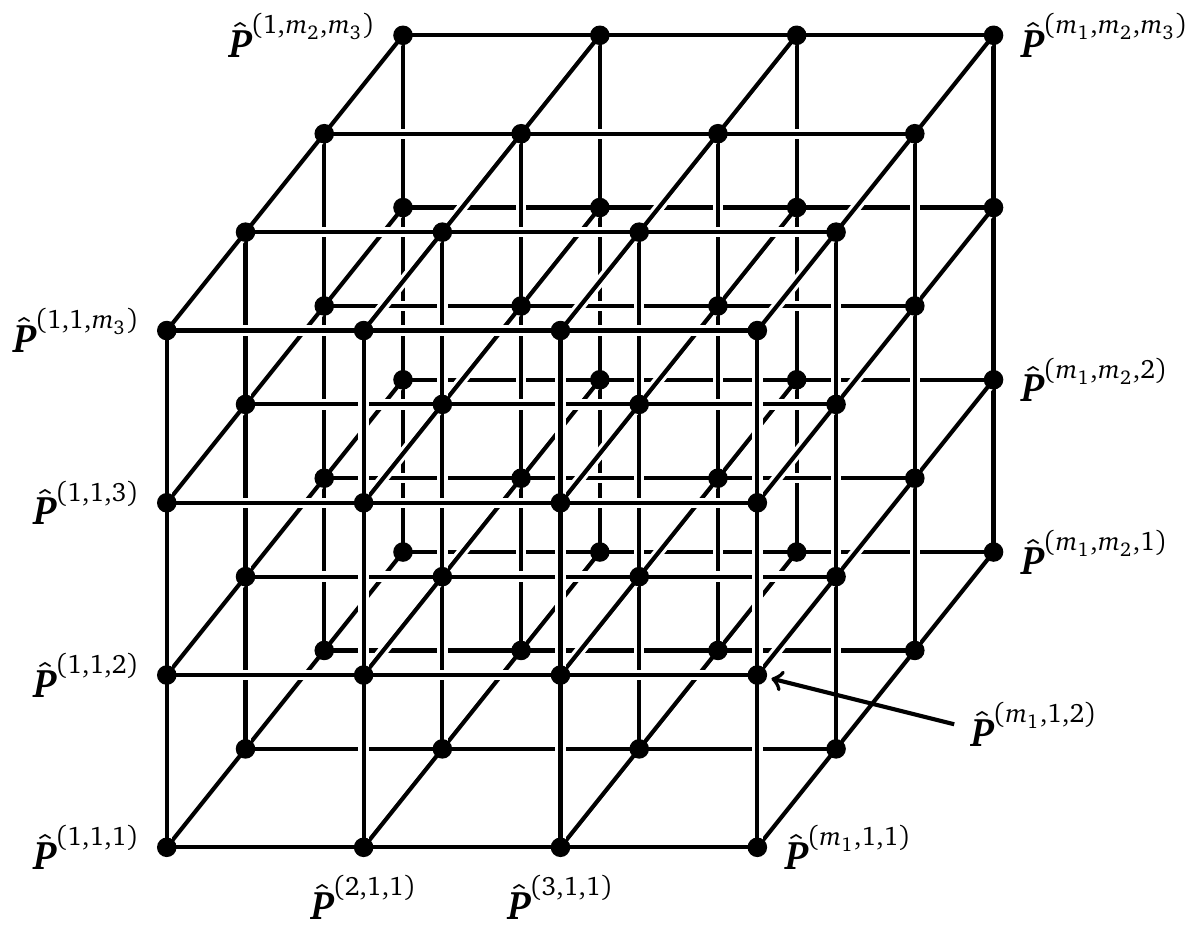}
  \caption[labelInTOC]{A spline box for $m_j=4, j=1,2,3,$ resulting in 64 control points.}
  \label{fig-spb}
\end{center}
\end{figure}

For the two-scale setting we also want to allow for rotation of the microscopic unit cell.
We introduce additional parameters $\thetabf\in [0;\tfrac{\pi}{2}]^3$ for rotation around the three coordinate axes.
However, we do not rotate the spline box directly but rather the homogenized coefficients by
\begin{equation}\label{eq-rot}
\Hop_r = \Rb(\thetabf)\; \Hop\; \Rb(\thetabf)^\top,
\end{equation}
where $\Rb$ is a corresponding rotation matrix.

\subsection{Regularity preserving injectivity conditions}
As already mentioned in the introduction, we do not want to remesh or refine the initial finite element mesh in our numerical examples.
Thus, we have to prevent the mesh from degeneration due to design changes and hence, we have to ensure injectivity of the mapping defined in \eq{eq-spl1}.
For this some notations as well as a few geometrical concepts are introduced.
Let $\imx_k^+ = \imx + e_k$ and $\imx_k^- = \imx - e_k$, $k=1,2,3$, where $e_1=(1,0,0),e_2=(0,1,0),e_3=(0,0,1)$.
For $k\in\{1,2,3\}$ we define the index set
\begin{equation}
I(k) = \{\imx\; |\; i_j=1,\ldots,m_j \text{ for } j\neq k \text{ and } i_k=1,\ldots,m_k-1\}\;.
\end{equation}
Following the approach in \cite{Xu2013} we define for $\imx\in I(k)$ the finite differences
\begin{equation}\label{eq-delta}
\Delta^\imx_k = \Pb^{\imx_k^+} - \Pb^{\imx}
\end{equation}
of two adjacent control points.
Then the first derivative of $\xb$ in \eq{eq-spl1} with respect to $t_k$ may be expressed as
\begin{equation}
\delta_{t_k} \xb = \sum_{\imx\in I(k)} \omega^\imx \Delta^\imx_k \Bcal^{\imx_k^-}(\tb)\;,
\end{equation}
where $\omega^\imx$ is a positive factor independent of the control points.
\begin{e-definition}
Let the \emph{convex hull} of a subset $X$ of a vector space $V$ be defined by
\begin{equation}
\conv X = \left\{ \sum_{i=1}^n \lambda_i x_i \;\big|\; x_i \in X, n\in\mathbb{N}, \sum_{i=1}^n \lambda_i=1, \lambda_i\geq 0 \right\}\;.
\end{equation}
The \emph{conical hull} of a subset $X$ of a vector space $V$ is given by
\begin{equation}
\coni X = \left\{ \sum_{i=1}^n \lambda_i x_i \;\big|\; x_i \in X, n\in\mathbb{N}, \lambda_i\geq 0 \right\}\;.
\end{equation}
A subset $X$ of a vector space $V$ is called \emph{cone}, if
\begin{equation}
x\in X, \lambda \in \RR, \lambda \geq 0 \Rightarrow \lambda x \in X\;.
\end{equation}
Finally we denote by $X^*$ the \emph{blunt cone} $X\setminus\{0\}$.
\end{e-definition}
Note that by this definition each cone or conical hull contains the point of origin and a conical hull is always a convex cone.

\begin{e-definition}\label{def-cot}
Two cones $C_1,C_2$ are said to be \emph{transverse} if $C_1\cup -C_1$ and $C_2\cup -C_2$ intersect only at $\{0\}$.
Three cones $C_1,C_2,C_3$ are said to be \emph{cotransverse} if the origin is a vertex of the convex hull of $C_1,C_2,C_3$ and the cone $C_i$ and the convex hull of the cones $C_j,C_k$ are transverse for all $\{i,j,k\}=\{1,2,3\}$.
\end{e-definition}
We now define cones generated by the forward finite differences of adjacent control points in each direction $k$ by
\begin{equation}\label{eq-ck}
C_k = \coni\{\Delta^\imx_k\; |\; \imx\in I(k)\}\;.
\end{equation}
Using the previous definitions we are able to claim the following statement.

\begin{e-proposition}\label{pro-inj}
Let the mapping \eq{eq-spl1} be at least $C^1$ and the boundary surfaces for fixed $\alphabf$ not intersect and have no self-intersection point.
Let further
\begin{itemize}
  \item[(A1)] $\Delta^\imx_k\neq 0$ for all $\imx, k$ and
  \item[(A2)] the cones $C_1,C_2,C_3$ be cotransverse.
\end{itemize}
Then the mapping is injective on $[0,1]^3$.
\end{e-proposition}

\begin{proof}
See \cite[Proposition 3.4]{Xu2013}.
\end{proof}
Next we want to derive sufficient conditions for $(A1)$ and $(A2)$.
As a consequence of $(A2)$ the blunt cones $C_k^* (k=1,2,3)$ are pairwise disjoint.
As the cones $C_k$ are also convex, we can conclude by the hyperplane separation theorem that there exist three planes separating the blunt cones $C_k^*$ from each other.
By the same argument we deduce the existence of planes separating one blunt cone from the convex hull of the other two.
This also holds for the mirrored cones $-C_k^*$ with the same planes as the separated sets just switch sides or signs in the formula, respectively.

To obtain sufficient conditions we first define four planes
\begin{equation}\label{eq-sp1}
\begin{split}
H_1(\xb) = x_1 - x_2 - x_3\;,\quad & H_2(\xb) = x_1 - x_2 + x_3\;, \\
H_3(\xb) = x_1 + x_2 - x_3\;,\quad & H_4(\xb) = x_1 + x_2 + x_3\;,
\end{split}
\end{equation}
and cones
\begin{equation}\label{eq-c0}
\begin{split}
C_1^0 &= \left\{\xb\in\RR^3 : H_1(\xb) > 0, H_2(\xb) > 0, H_3(\xb) > 0, H_4(\xb) > 0\right\} \cup \{0\}\,, \\
C_2^0 &= \left\{\xb\in\RR^3 : H_1(\xb) < 0, H_2(\xb) < 0, H_3(\xb) > 0, H_4(\xb) > 0\right\} \cup \{0\}\,, \\
C_3^0 &= \left\{\xb\in\RR^3 : H_1(\xb) < 0, H_2(\xb) > 0, H_3(\xb) < 0, H_4(\xb) > 0\right\} \cup \{0\}\,,
\end{split}
\end{equation}
see also figure \ref{fig-cot}.

\begin{figure}[tp]
\begin{center}
  \begin{subfigure}[t]{.45\textwidth}
    \includegraphics[height=5.5cm]{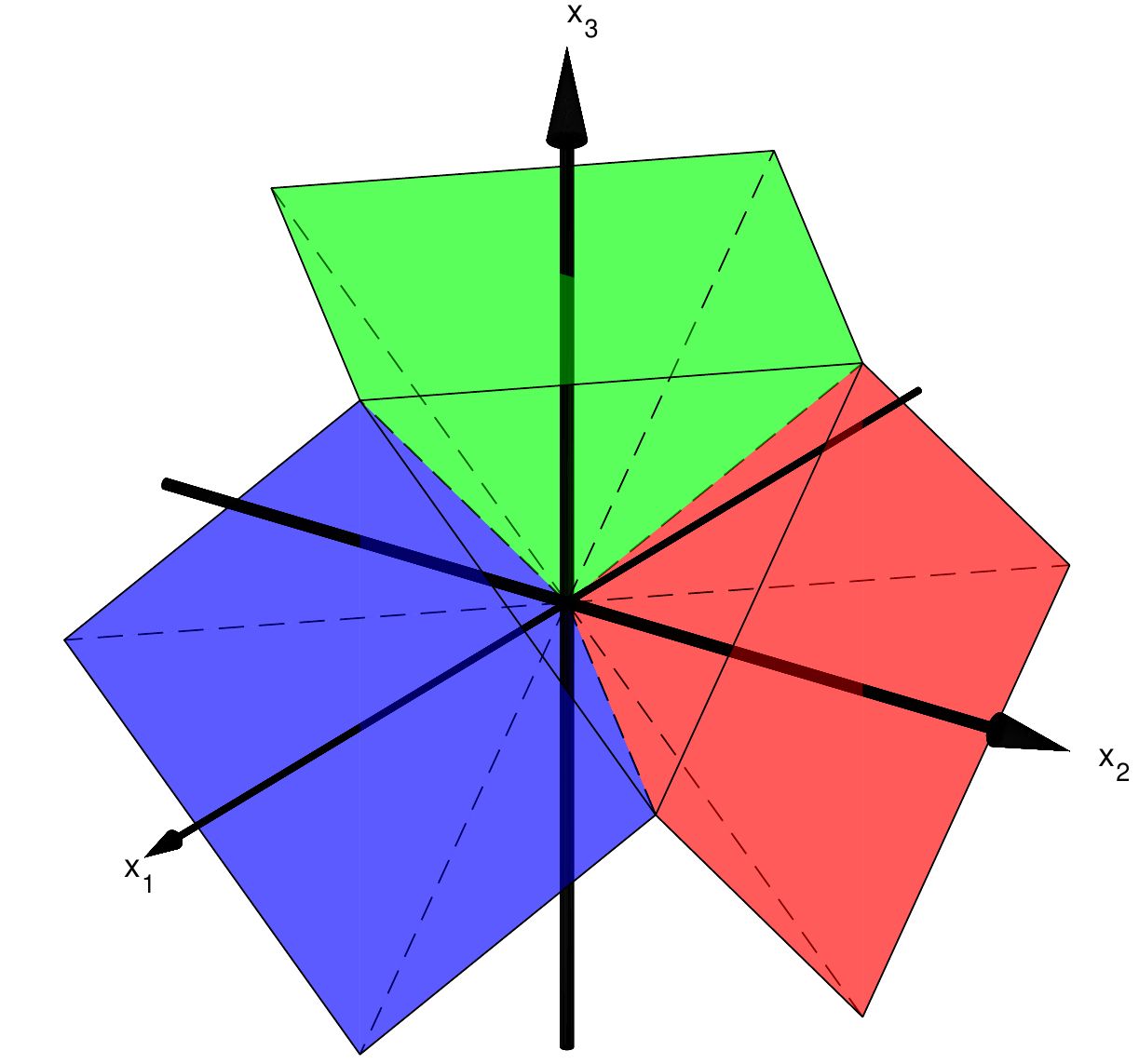}
  \end{subfigure}
  \begin{subfigure}[t]{.45\textwidth}
    \includegraphics[height=5.5cm]{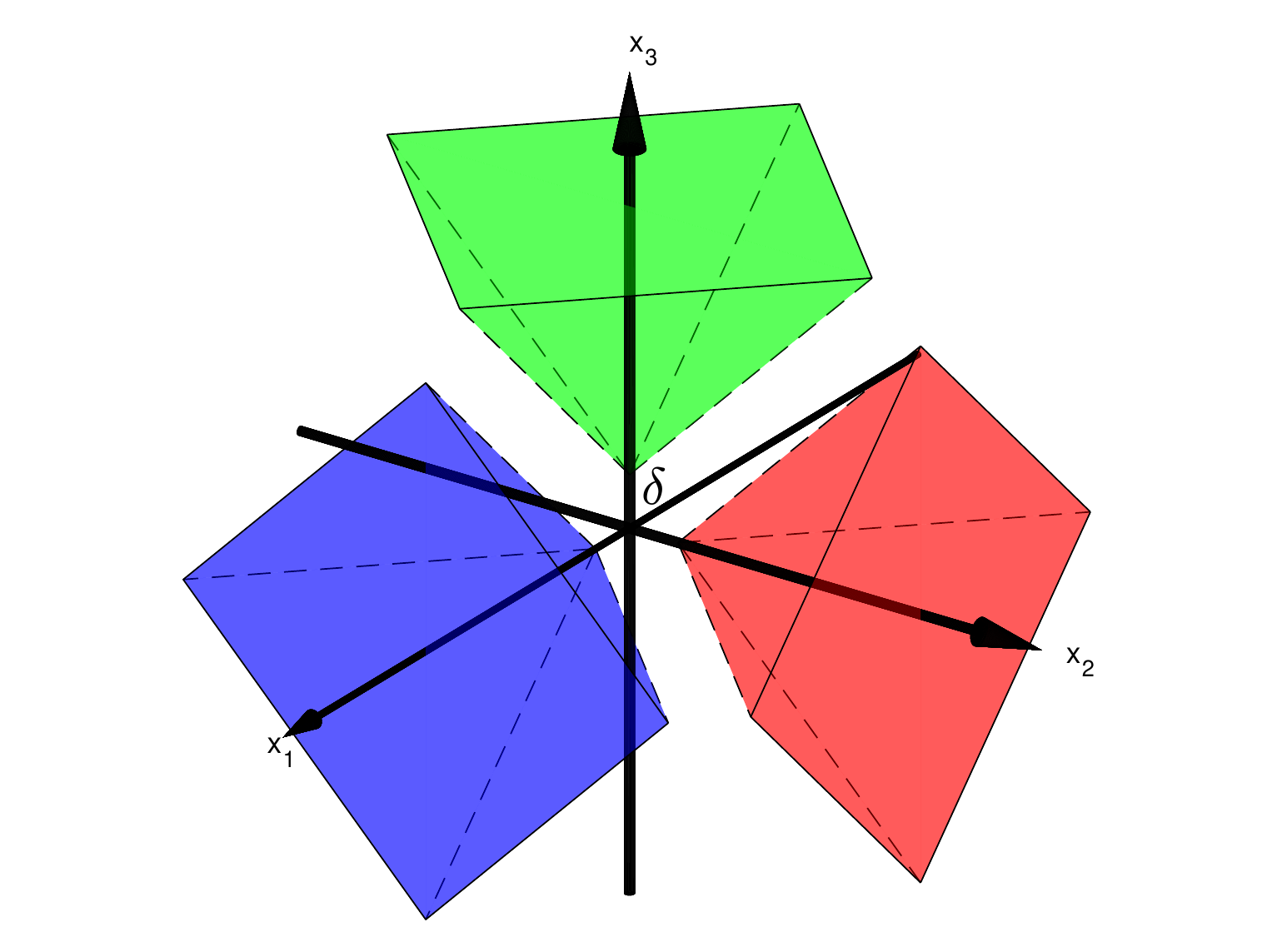}
  \end{subfigure}
  \caption[labelInTOC]{left: the cotransverse cones $C_1^0$ (blue), $C_2^0$ (red) and $C_3^0$ (green). right: the regularized cones with parameter $\delta$.}
  \label{fig-cot}
\end{center}
\end{figure}

\begin{lemma}\label{lem-co}
The cones $C_1^0,C_2^0,C_3^0$ are cotransverse.
\end{lemma}
\begin{proof}
It is easy to see that the origin is a vertex of the convex hull of the cones
\begin{equation*}
\begin{split}
\conv\{C_1^0,\: &C_2^0, C_3^0\} = \\
&\left\{\xb\in\RR^3 : H_1(\xb) \neq 0, H_2(\xb) \neq 0, H_3(\xb) \neq 0, H_4(\xb) > 0\right\} \cup \{0\}\;.
\end{split}
\end{equation*}
Now, without loss of generality let $i=1$.
Let further
\begin{equation*}
\begin{split}
S := &\conv\{ C_2^0, C_3^0\} \\
= &\left\{\xb\in\RR^3 : H_1(\xb) < 0, H_2(\xb) \neq 0, H_3(\xb) \neq 0, H_4(\xb) > 0\right\} \cup \{0\} \;.
\end{split}
\end{equation*}
Then $(C_1^0\cup -C_1^0) \cap (S\cup -S) = \{0\}$ and thus $C_1^0$ and $S$ are transverse.
For $i\in\{2,3\}$ we get a similar result and hence $C_1^0, C_2^0, C_3^0$ are cotransverse.
\end{proof}

Then obviously the cones $C_1,C_2,C_3$ defined by \eq{eq-c0} are cotransverse if
\begin{equation}\label{eq-incl}
C_k\subset C_k^0,\ k=1,2,3\;.
\end{equation}
\begin{lemma}\label{lem-cot}
$C_1, C_2, C_3$ are cotransverse if
\begin{equation}\label{eq-convcot}
\conv\{\Delta^\imx_k\; |\; \imx\in I(k)\} \subset C_k^0,\ k=1,2,3\;.
\end{equation}
\end{lemma}
\begin{proof}
For $\imx\in I(k)$ and $k\in\{1,2,3\}$ we have
\begin{equation*}
C_k = \coni\{\Delta^\imx_k\; |\; \imx\in I(k)\} = \coni\conv\{\Delta^\imx_k\; |\; \imx\in I(k)\} \subset \coni C_k^0 = C_k^0\;,
\end{equation*}
i.\;e. \eq{eq-incl} holds.
\end{proof}

We note that conditions \eq{eq-incl} are sufficient but not necessary for $(A2)$ to hold.
For example they forbid rotations of the initial spline box whereas the cones still would be cotransverse.
However they can be realized by linear inequalities which we will see later on.
The latter are not only easy to handle in an optimization algorithm but also there exist algorithms assuring strict feasibility during the optimization procedure up to a certain tolerance.
But still the constraint \eq{eq-convcot} cannot be treated by standard optimization software directly as the cones $C_k^0$ $(k=1,2,3)$ are not closed.
As a remedy we introduce a regularization parameter $\delta > 0$ and define
\begin{equation}\label{eq-c0-delta}
\begin{split}
C_1^\delta &= \left\{\xb\in\RR^3 : H_1(\xb) \geq \phantom{-}\delta, H_2(\xb) \geq \phantom{-}\delta, H_3(\xb) \geq \phantom{-}\delta, H_4(\xb) \geq \delta\right\}\,, \\
C_2^\delta &= \left\{\xb\in\RR^3 : H_1(\xb) \leq -\delta, H_2(\xb) \leq -\delta, H_3(\xb) \geq \phantom{-}\delta, H_4(\xb) \geq \delta\right\}\,, \\
C_3^\delta &= \left\{\xb\in\RR^3 : H_1(\xb) \leq -\delta, H_2(\xb) \geq \phantom{-}\delta, H_3(\xb) \leq -\delta, H_4(\xb) \geq \delta\right\}\,.
\end{split}
\end{equation}
\begin{corollary}\label{lem-cot2}
If $\conv\{\Delta^\imx_k\; |\; \imx\in I(k)\} \subseteq C_k^\delta$, $k=1,2,3$, then $(A1)$ and $(A2)$ hold.
\end{corollary}
This follows easily from $C_k^\delta \subset C_k^0$ and $0 \notin C_k^\delta$.
The relaxation parameter $\delta$ shifts the cones $C_1^0, C_2^0, C_3^0$ away from the point of origin (see figure \ref{fig-cot}).
\begin{remark}\label{rem3}
Due to the performed numerical experiments, we observe that during the optimization process, the finite element mesh is less prone to its degeneration with increasing $\delta$.
It can be shown that $\delta$ gives direct control over the minimal distance of two adjacent control points.
Owing to the continuity of the map from the control points onto the finite element nodes \eqref{eq-spl1}, there is also a control over the maximal deformation of a finite element.
However, a quantitative result remains to be shown.
\end{remark}

Now we state the linear constraints derived from \eq{eq-incl}.
For $\Delta^\imx_k \in C_k$ with $k=1,2,3$ and $\imx\in I(k)$ we have four inequalities of type
\begin{equation}
\mu_{s,k} H_s(\Delta^\imx_k) \geq \delta\;,
\end{equation}
where $\mu_{s,k} \in \{-1;1\}$ is chosen appropriatly to match the inequality to one of those in the definition of $C_k^\delta$ \eq{eq-c0-delta}.
Recalling definition \eq{eq-delta} and $\eq{eq-spl1}_2$
\begin{equation*}
\begin{split}
\Delta^\imx_k = \Pb^{\imx_k^+} - \Pb^{\imx}\;, \\
\Pb^{\,\imx}(\alphabf^\imx) = \hat\Pb^{\,\imx} + \alphabf^\imx
\end{split}
\end{equation*}
and exploiting the linearity of $H_s$ we finally get linear constraints for the design variables $\alphabf$ of form
\begin{equation}\label{eq-inj}
-\mu_{s,k} H_s(\alpha^{\imx_k^+} - \alpha^\imx) \leq \mu_{s,k} H_s(\hat\Pb^{\imx_k^+} - \hat\Pb^{\,\imx}) - \delta\;.
\end{equation}
These conditions are sufficient for injectivity of \eq{eq-spl1} for all deformations $\alphabf$.
Hence applying them in the optimization problems described in the following sections we are able to prevent degeneration of the unit cell $Y$ in each iteration of the optimization algorithm.

\begin{remark}
We note that the choice of the planes $H_s$ may strongly influence optimal solutions.
Thus it might be better to choose them in a more problem specific way, for example widen cones and narrow others in exchange.
\end{remark}

\section{Numerical results}\label{sec-numres}
\FloatBarrier
In this section, we present numerical examples illustrating the two kinds of optimization problems introduced in Sections~\ref{sec-matopt} and~\ref{sec-2scopt}.
For both we consider the same initial design on the porous microstructure.
With reference to the representative cell $Y$ split according to \eq{eq-6}, the pore $Y_c$ consists of 3 perpendicular cylindrical channels $Y_c^i$, $i = 1,2,3$ forming a rectangular cross.
At the intersection the channel axes, a sphere $Y_c^4$ is placed, so that the periodic porosity is generated by $Y_c = \bigcup_{i = 1,\dots,4} Y_c^i$.

\begin{table}
  \begin{center}
    \begin{tabular}{l|c|c|c|c|c|c|r|c|r}
      problem & fig. & $r$ & $\kappa_0$ & $\kappa_1$ & $s_0$ & $s_1$ & \#iter. & objective & rel. gain \\
      \hline
      \textit{S/P} & \ref{fig-SP} & 0 & 2e-5 & $-\infty$ & - & - & 101 & 1.149 & 7.3\% \\
      \textit{S/P-bis} & \ref{fig-SPbis} & 1 & 2e-5 & $-\infty$ & - & - & 35 & 1.082 & 1.0\% \\
      \textit{S/PX} & \ref{fig-SPX} & 0 & $-\infty$ & 2e-5 & - & - & 159 & 1.192 & 11\% \\
      \textit{P/SX'} & \ref{fig-PSX} & 0 & - & - & 0.9 & 0.95 & 114 & 12.43e-5 & 330\% \\
      \textit{C/S} & \ref{fig-CS} & 0 & - & - & $\Phi_\eb(\Aop_0)$ & - & 41 & 4.298 & 2.2\% \\
    \end{tabular}
    \caption{The numerical examples for material optimization. For the problem definitions including explanations for the parameters $r, \kappa_0, \kappa_1, s_0, s_1$ see section \ref{sec-matopt}.}
    \label{tab-num-ex}
  \end{center}
\end{table}


We apply a cubic spline box with three segments of degree three in each direction resulting in 216 control points (cf. figure \ref{fig-initial}) to parameterize the shape of $Y_c$, noting that also the shape of $Y$ can be modified within the periodicity constraints, as discussed in Section~\ref{sec-param}.
Thus, after elimination of degrees of freedom by periodic boundary conditions, the free coordinates of the control points $\{\alphabf^i\}$ yield 375 design variables.

\begin{figure}[tp]
  \begin{center}
    \begin{subfigure}[t]{.53\textwidth}
      \includegraphics[trim={0 90pt 0 90pt},clip,width=.9\columnwidth]{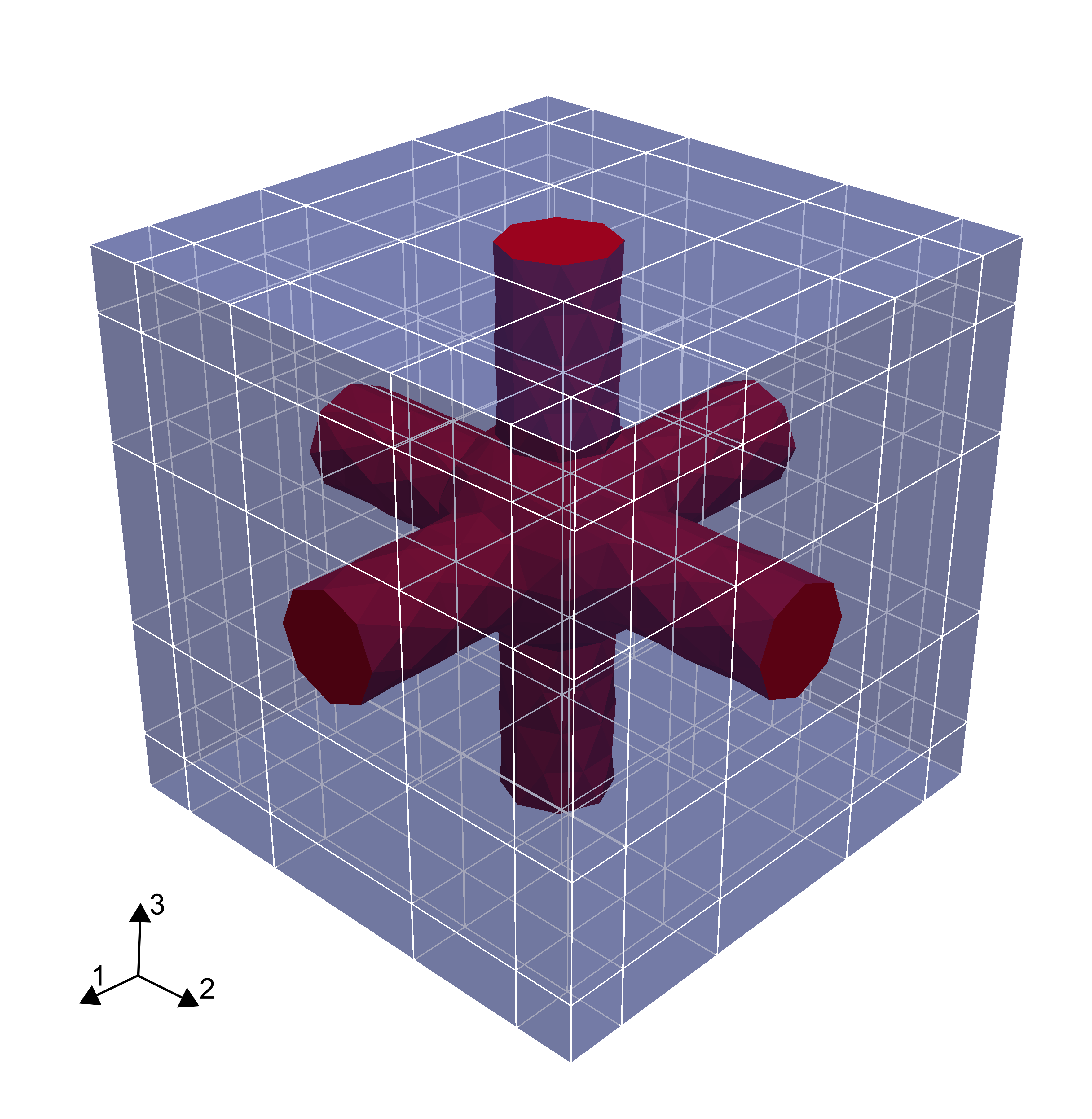}
      \label{fig-initial-a}
    \end{subfigure}
    ~
    \begin{subfigure}[t]{.43\textwidth}
      \includegraphics[trim={150pt 90pt 30pt 30pt},clip,width=\columnwidth]{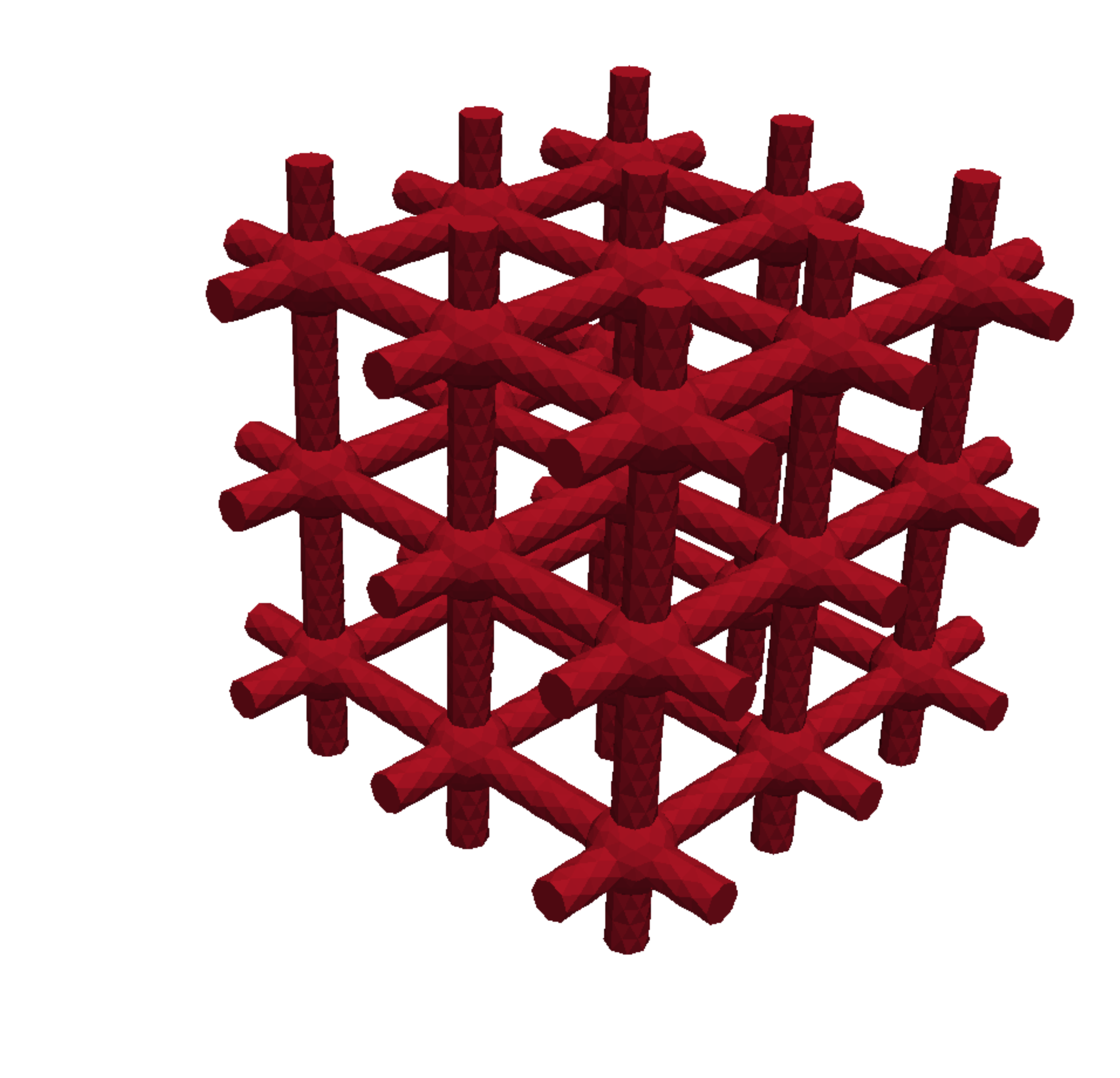}
      \label{fig-initial-b}
    \end{subfigure}
    \caption{Left: The initial setup: a solid cube (blue) with three channels (red) intersecting at a hollow sphere (red).
      The spline box has three segments of degree three in each direction resulting in 216 control points with the control polygon shown in white.
      right: 3x3x3 repeated channel part}
    \label{fig-initial}
  \end{center}
\end{figure}

\subsection{Material optimization}

For $k=1,2,3$ we choose $\gb^k = e_k$ unit vectors and $\eb^k$ and $\sigmabf^k$ such that for the initial design $\Phi_\eb^k(\Aop) = \Aop_{kkkk}$ and $\Phi_\sigmabf^k(\Cop) = \Cop_{kkkk}$.
The stiffness of the initial design is $\Phi_\eb^k(\Aop) = 1.071$ or $\Phi_\sigmabf^k(\Cop) = 4.206$, respectively, and the permeability is $\Psi^k(\Kb) = 4.3\cdot 10^{-5}$ in each direction $k=1,2,3$.

For the optimization we set the bounds for the directional permeabilities $\kappa_0 = 2\cdot 10^{-5}$, and also the geometry regularization parameter introduced in \eq{eq-inj},
 $\delta = 0.02$.
We solve the optimization problems using the sparse nonlinear optimizer SNOPT.
In contrast to Fig.~\ref{fig-initial}, we omit visualization of the control polygon in all forthcoming figures.

For problems \textit{S/P}, \textit{S/P-bis}, \textit{S/PX} we choose $\gamma^k = \tfrac13, k=1,2,3$.
The achieved solution of problem \textit{S/P} is shown in Fig.~\ref{fig-SP}.
The optimized design follows our intuition how to maximize the stiffness; diamteres of the channels in all the three directions and the sphere are made as small as possible, thus, allowing for increasing the skeleton stiffness.
The channels still form a rectangular cross, since no  direction is preferred.
The achieved stiffness is $\Phi_\eb(\Aop) = 1.149$, \ie improvement by $7.3\%$, whereby the permeability constraints are all active at their lower bounds.

\begin{figure}[tp]
  \begin{center}
    \begin{subfigure}[t]{.53\textwidth}
      \includegraphics[trim={0 40pt 0 40pt},clip,width=\columnwidth]{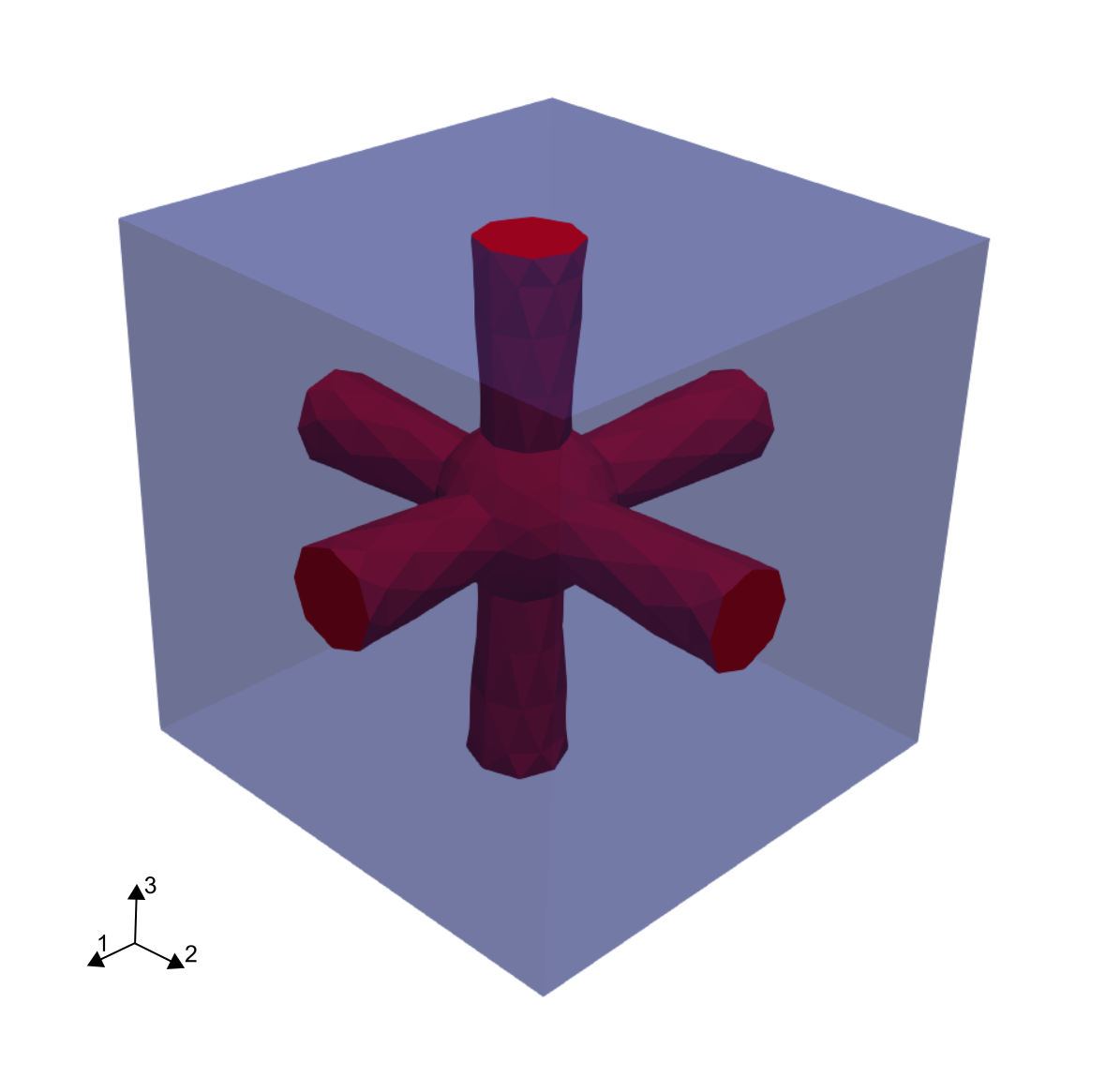}
      \label{fig-SP-a}
    \end{subfigure}
    ~
    \begin{subfigure}[t]{.43\textwidth}
      \includegraphics[trim={120pt 90pt 0 30pt},clip,width=\columnwidth]{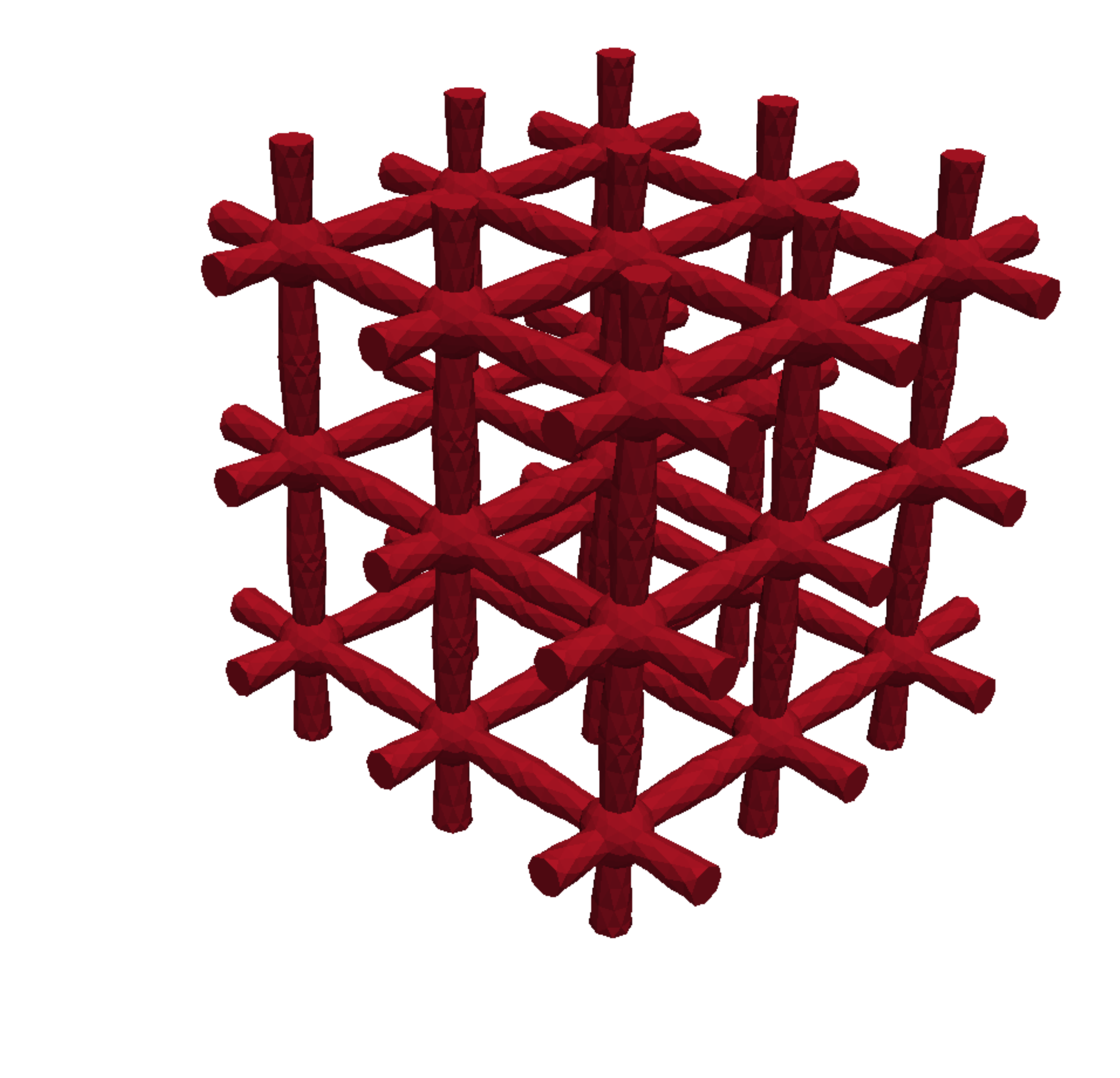}
      \label{fig-SP-b}
    \end{subfigure}
    \caption{Left: The optimized unit cell geometry regarding Problem \textit{S/P}. The objective was improved by $7.3\%$. right: 3x3x3 repeated channel part}
    \label{fig-SP}
  \end{center}
\end{figure}

For problem \textit{S/P-bis}, the resulting optimized microstructure is displayed in Fig.~\ref{fig-SPbis}.
In consequence of the considered volume preservation of $Y_c$, the channels get wider near to the boundary of the unit cell.
In fact the diameter of the channels is approximately equal to the diameter of the deformed sphere.
Also in this case the optimized redistribution of the solid material has a rather intuitive explanation -- since the channel junction is the weakest area in the cell $Y$, therefore, the pore volume is reduced as much as possible (by virtue of the channels diameters) within the other design constraints.
The stiffness of the optimized  microstructure improved by $1.0\%$ only, attaining $\Phi_\eb(\Aop) = 1.082$,
The permeability constraint for $k=3$, and the volume constraint are active, however,  permeability constraints $\Psi^k(\Kb) = 2.9\cdot 10^{-5}$ for $k=1,2$ are nonactive, which corresponds to the nonsymmetric design.
It should be noted at this point, that the deformation of the boundary of the unit cell is due to numerical reasons but only has little influence on the resulting design, as the cell is considered to be repeated periodically.

\begin{figure}[tp]
  \begin{center}
    \begin{subfigure}[t]{.53\textwidth}
      \includegraphics[trim={0 20pt 0 30pt},clip,width=.95\columnwidth]{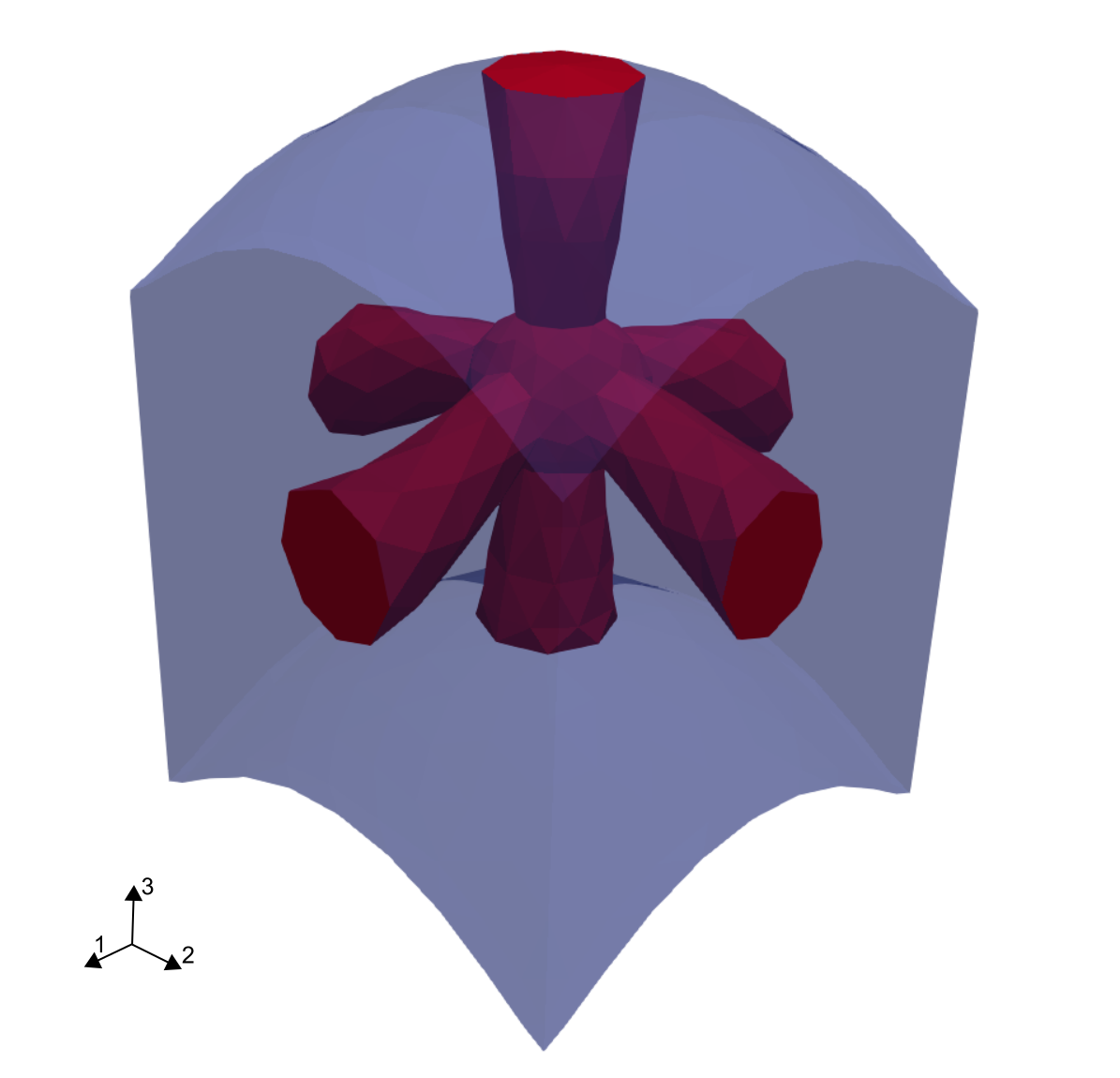}
      \label{fig-SPbis-a}
    \end{subfigure}
    ~
    \begin{subfigure}[t]{.43\textwidth}
      \includegraphics[trim={90pt 50pt 0 0pt},clip,width=\columnwidth]{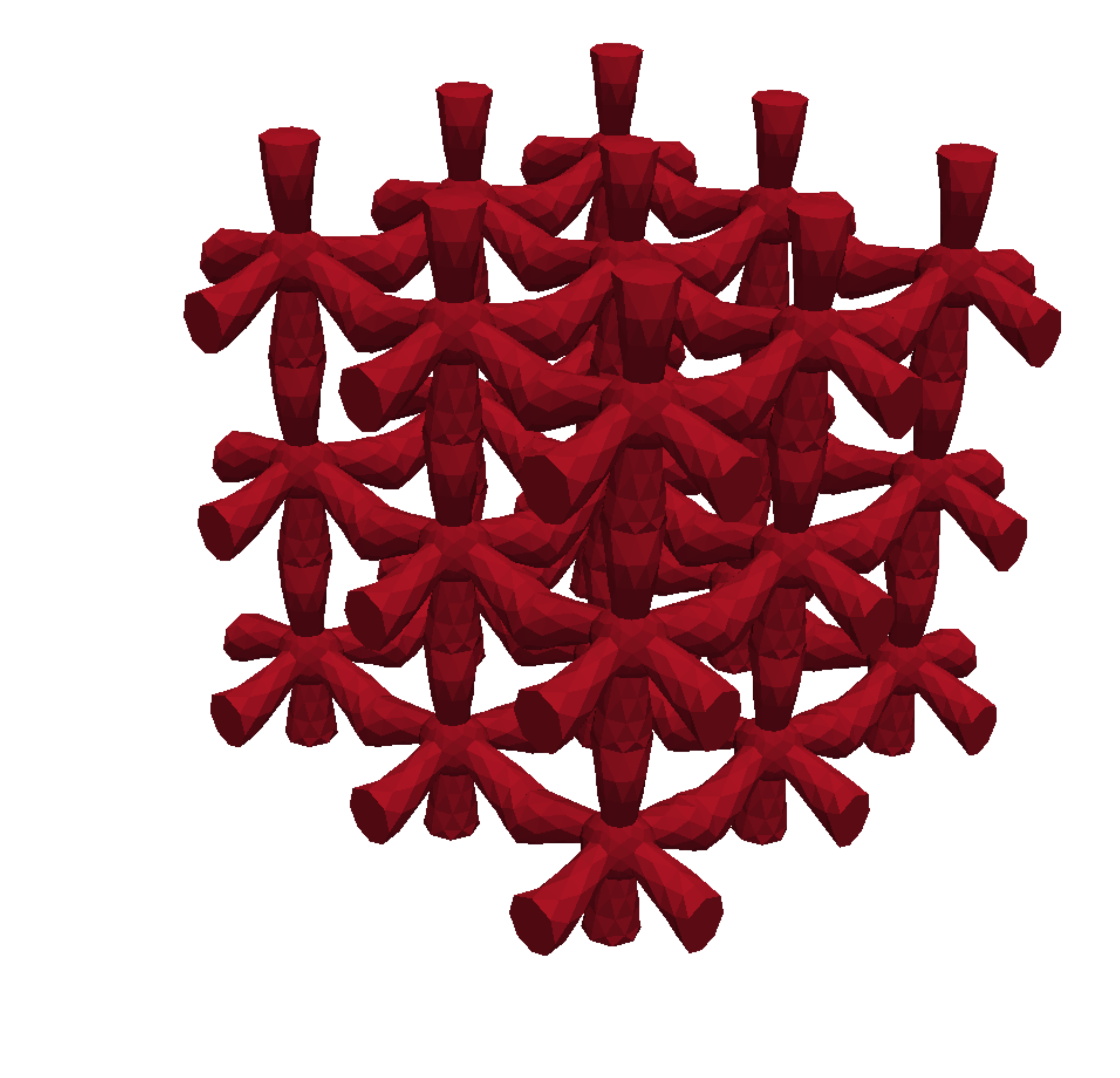}
      \label{fig-SPbis-b}
    \end{subfigure}
    \caption{The optimization result of problem \textit{S/P-bis}. The stiffness was improved by $1.0\%$.
      Note the widening of the channels at the boundaries due to the forced volume preservation of the channel part.}
    \label{fig-SPbis}
  \end{center}
\end{figure}

\begin{figure}[tp]
  \begin{center}
    \begin{subfigure}[t]{.48\textwidth}
      \includegraphics[width=\columnwidth]{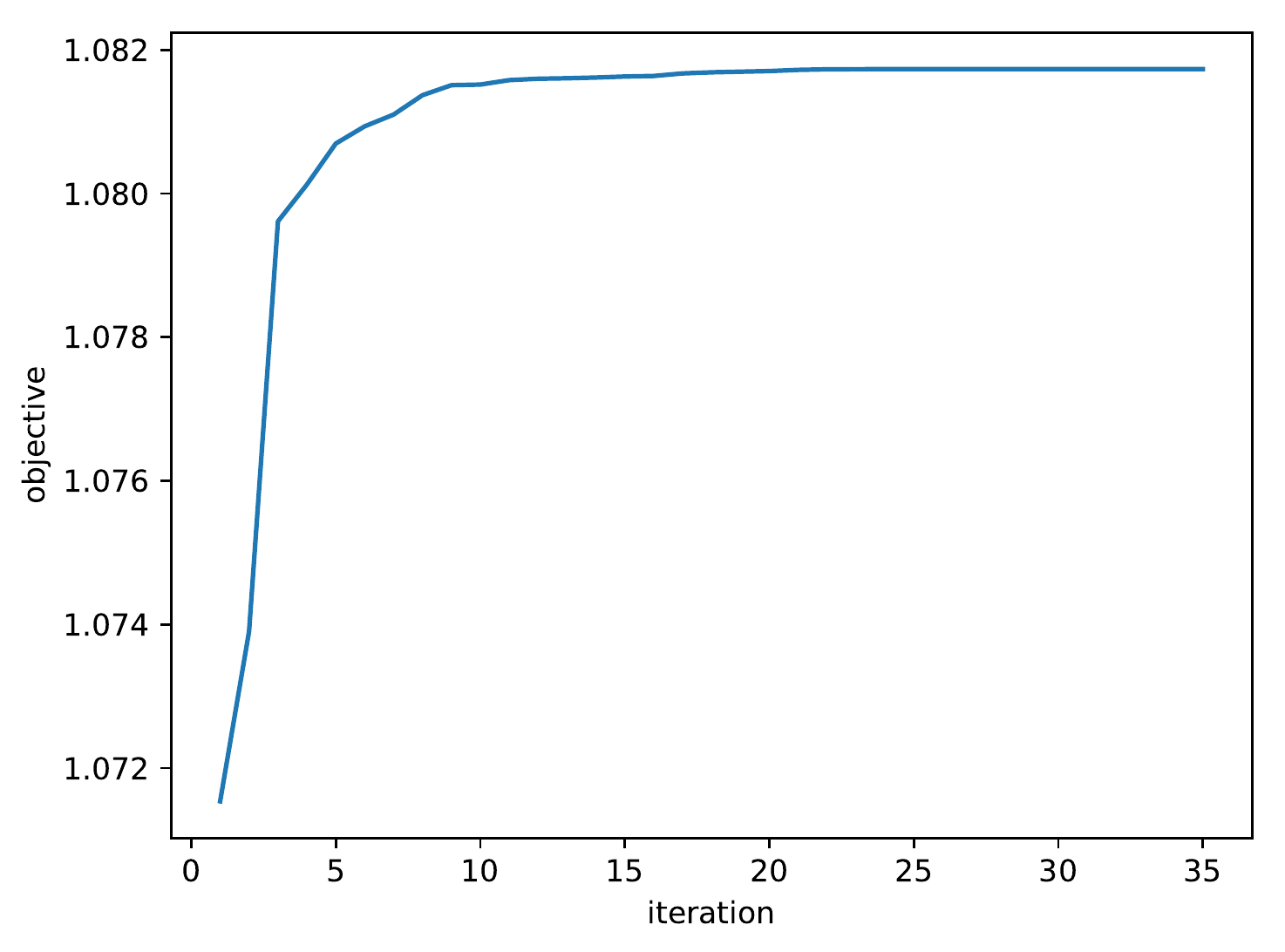}
      \label{fig-SPbis-convergence}
    \end{subfigure}
    ~
    \begin{subfigure}[t]{.48\textwidth}
      \includegraphics[width=\columnwidth]{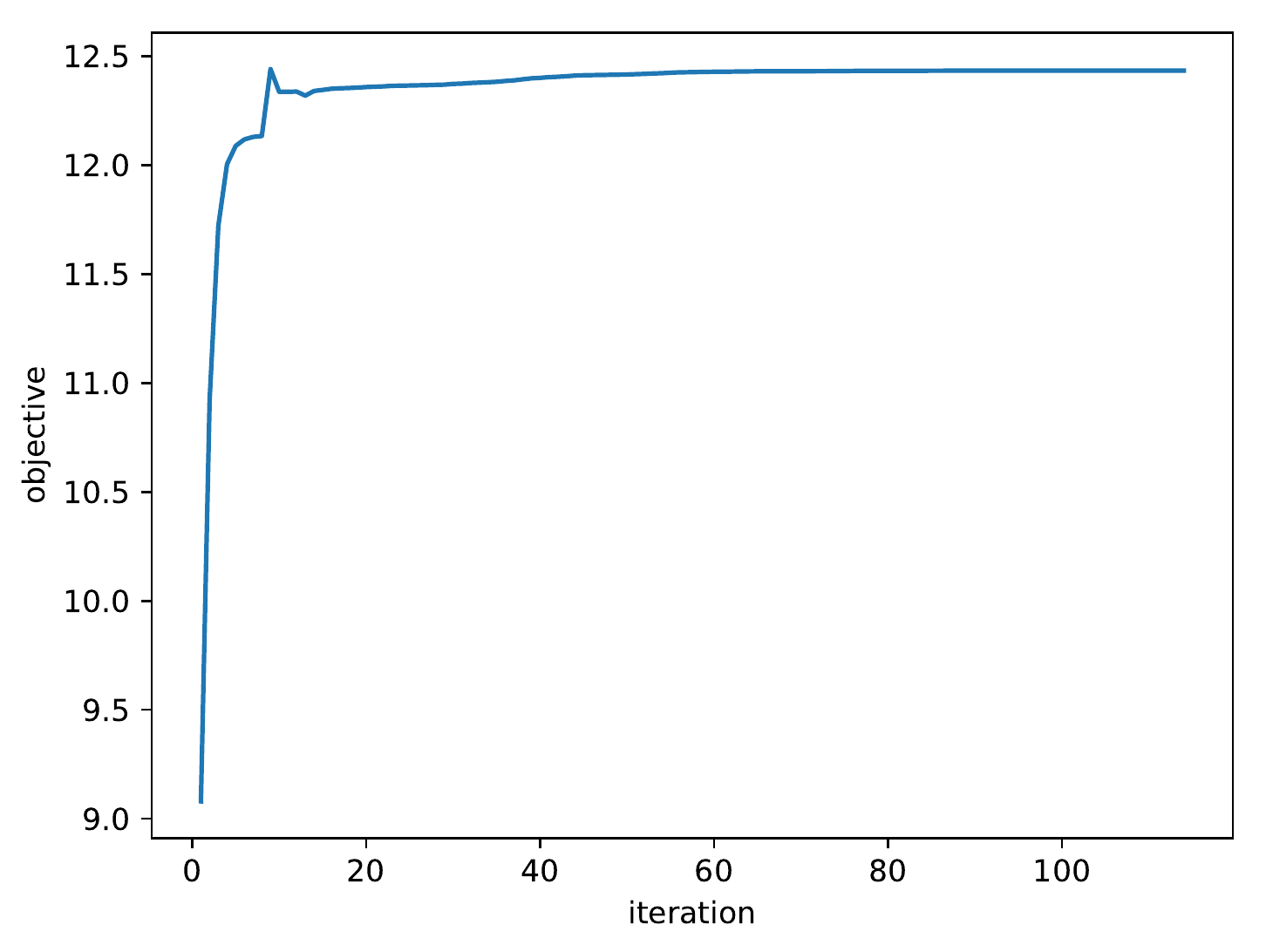}
      \label{fig-PSX-convergence}
    \end{subfigure}
    \caption{Convergence plots for problems \textit{S/P-bis} (left) and \textit{P/SX} (right)}
    \label{fig-matopt-convergence}
  \end{center}
\end{figure}

In all the previous examples the bounds on the permeability constraints along the principal axes $\gb^k$ as well as the weights $\gamma^k$ in the objective function $\Phi_\eb$ were chosen completely symmetric enforcing almost isotropic designs.
For problem \textit{S/PX}, see \eq{eq-op2},  we set $\beta^1=\beta^2=\tfrac49, \beta^3=\tfrac19$ and $\kappa_1 = 2\cdot 10^{-5}$.
Due to the different weights of the directional permeabilities, the optimization results in a remarkably anisotropic design, as can be seen in Fig.~\ref{fig-SPX}.
The permeability in the first direction is more important than the one in the third direction, thus the channel in the first direction is much wider than the one in the third direction.
Also a wider channel would be needed to increase the permeability in the second direction.
However, this is not possible because of to the chosen design parametrization.

\begin{figure}[tp]
  \begin{center}
    \begin{subfigure}[t]{.53\textwidth}
      \includegraphics[trim={0 40pt 0 40pt},clip,width=\columnwidth]{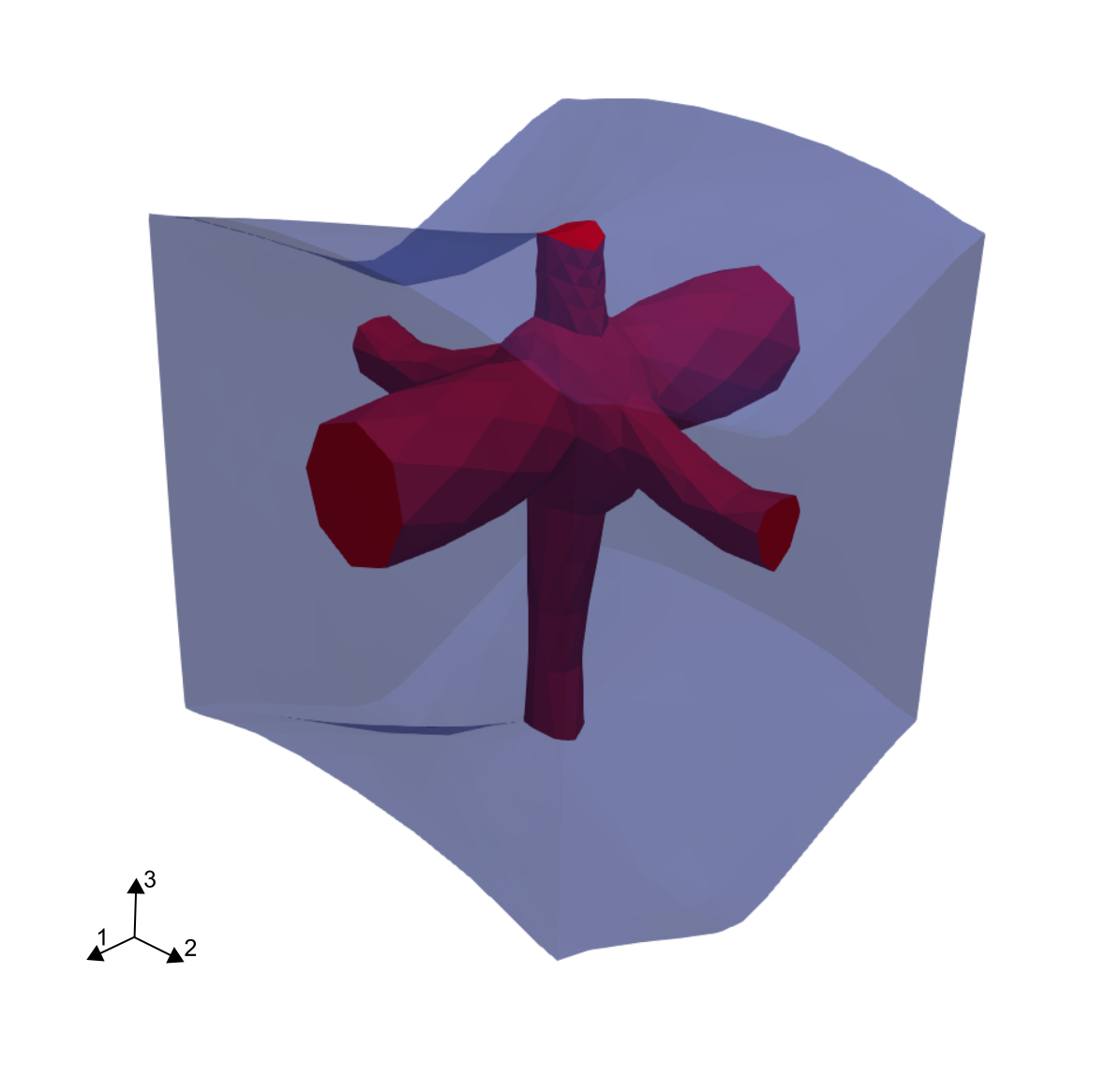}
      \label{fig-SPX-a}
    \end{subfigure}
    ~
    \begin{subfigure}[t]{.43\textwidth}
      \includegraphics[trim={100pt 30pt 30pt 30pt},clip,width=\columnwidth]{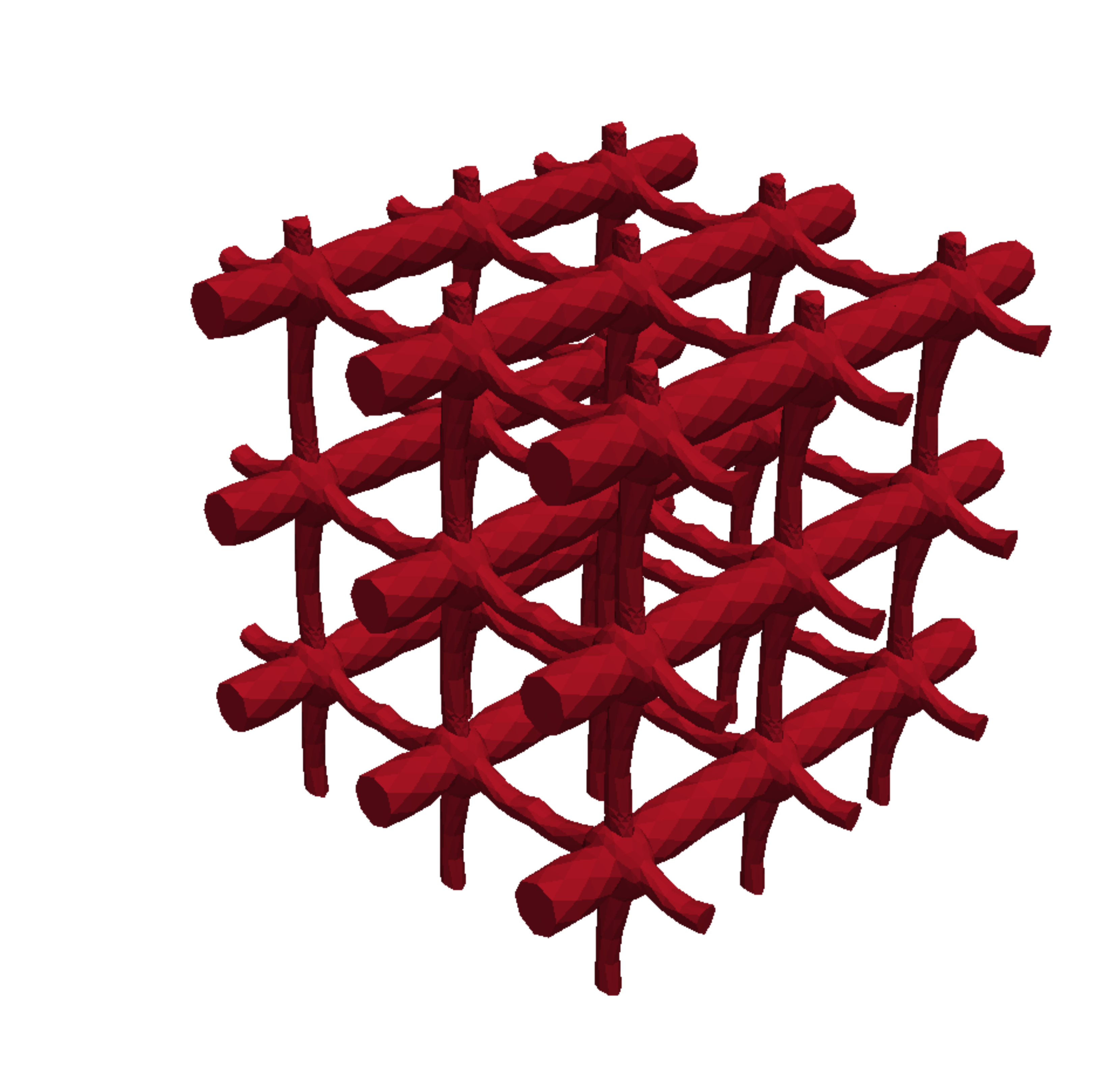}
      \label{fig-SPX-b}
    \end{subfigure}
    \caption{The optimized geometry for problem \textit{S/PX}. The stiffness was improved by $11\%$.}
    \label{fig-SPX}
  \end{center}
\end{figure}

An optimization result of problem \textit{P/SX'}, see \eq{eq-op3}, is shown in Fig.~\ref{fig-PSX}.
We choose $\beta^k = \tfrac13, k=1,2,3$ as weights for the permeability sum, $s_0=0.9$ and $s_1=0.95$ together with the weights for the stiffness sum $\gamma^1 = \tfrac{8}{13}, \gamma^2 = \tfrac{4}{13}, \gamma^3 = \tfrac{1}{13}$.
The choice $s_0 \approx s_1$ leads to almost symmetric design of the microstructure, although
 the stiffnesses in the preferred directions reflect the decreasing sequence of $\gamma^k$, $k=1,2,3$, namely $\Phi_\eb^1(\Aop) = 0.96, \Phi_\eb^2(\Aop) = 0.94, \Phi_\eb^3(\Aop) = 0.90$.
Also in this case,  to increase the stiffness in direction $\eb^1$, the sphere in the middle is stretched in this direction.
The achieved objective function value is $\Psi(\Kb) = 12.4\cdot 10^{-5}$ gives rise to an improvement by $328\%$ with respect to the initial design.

\begin{figure}[tp]
  \begin{center}
    \begin{subfigure}[t]{.53\textwidth}
      \includegraphics[trim={0 40pt 0 90pt},clip,width=\columnwidth]{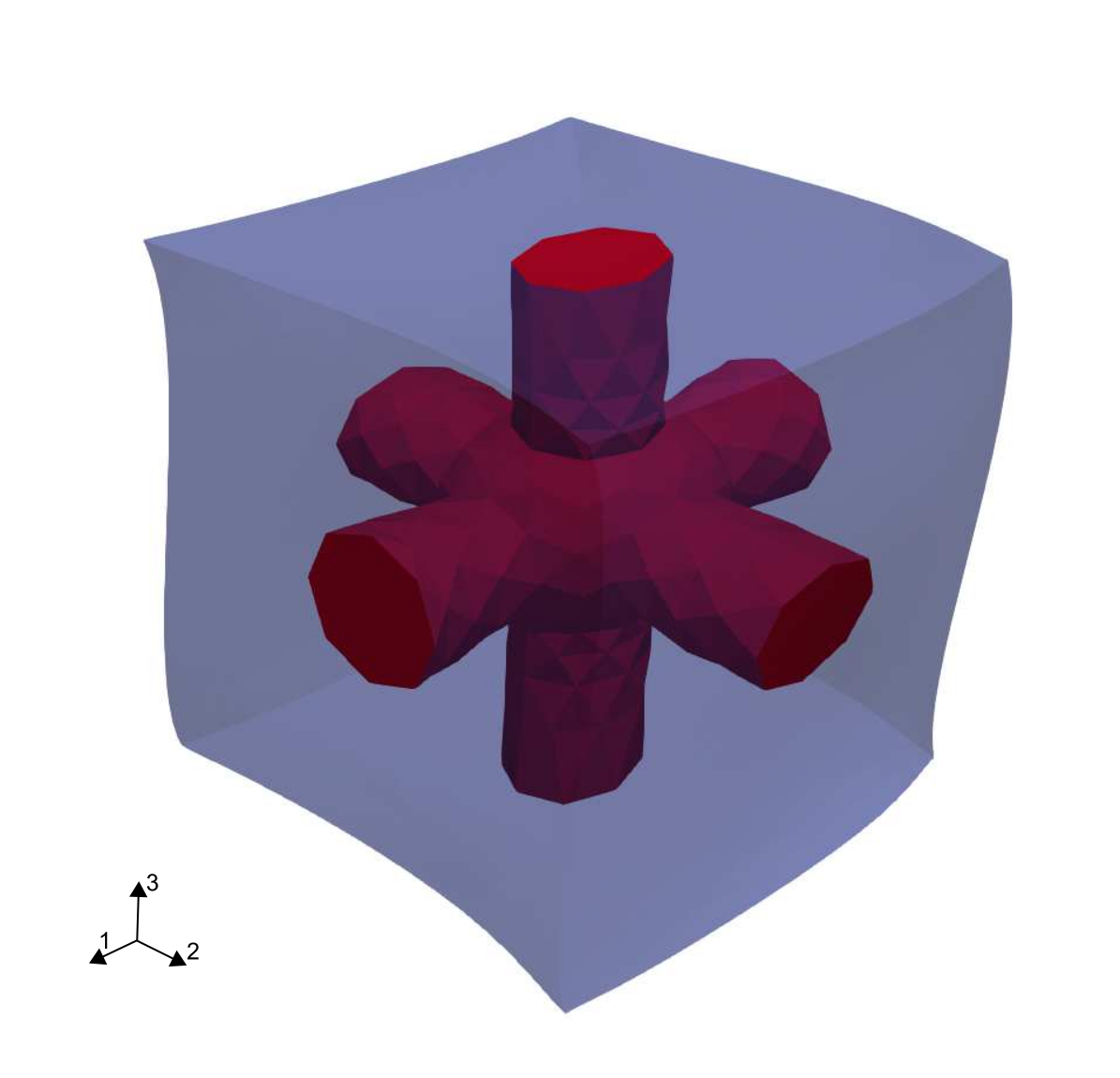}
      \label{fig-res1c-a}
    \end{subfigure}
    ~
    \begin{subfigure}[t]{.43\textwidth}
      \includegraphics[trim={100pt 60pt 30pt 30pt},clip,width=\columnwidth]{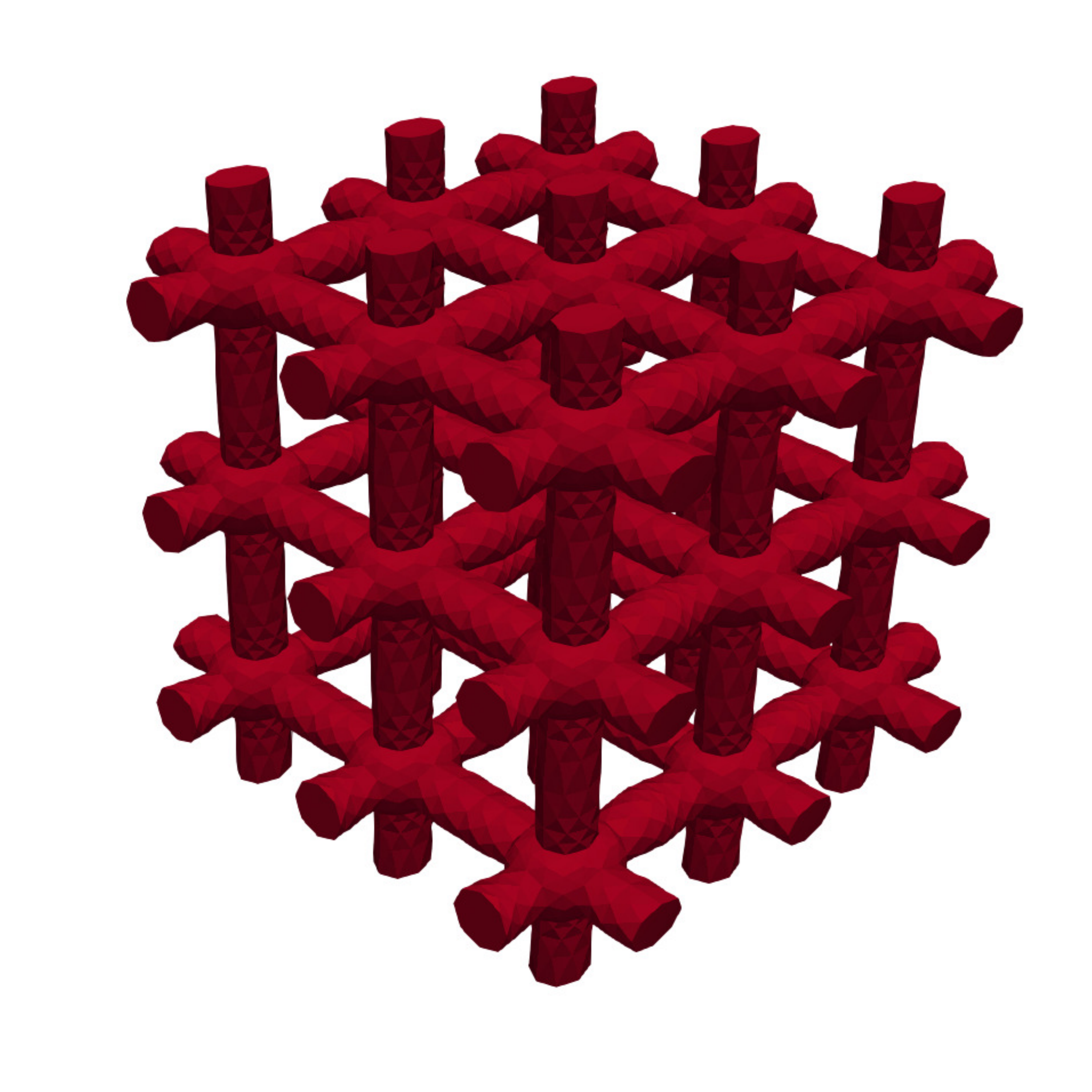}
      \label{fig-res1c-b}
    \end{subfigure}
    \caption{The optimization result of problem \textit{P/SX'}. The permeability was improved by $328\%$.}
    \label{fig-PSX}
  \end{center}
\end{figure}

To solve the minimization problem \textit{C/S}, see \eq{eq-op5}, we set $s_0=\Phi_\eb(\Aop_0)$, where $\Aop_0$ is the stiffness tensor of the drained initial design again with $\gamma^k=\tfrac13, k=1,2,3$.
In this case, since the porous material permeability is irrelevant, the undrained stiffness maximization induces redistribution of the solid phase so that the sphere in the middle of the cell $Y$ is blown up whilst the channels are getting rather smaller near to the (periodic) boundary $\pd Y$, see Fig.~\ref{fig-CS}.
This shape modification reflects the trend to create one spherical inclusion; note that the drained stiffnesses (as measured in directions $k = 1,2,3$) is bounded due to the imposed constraints.
Compared to the initial layout, the undrained compliance, \ie the objective, decreased  merely by $2.1\%$.

\begin{figure}[tp]
  \begin{center}
    \begin{subfigure}[t]{.53\textwidth}
      \includegraphics[trim={0 40pt 0 80pt},clip,width=\columnwidth]{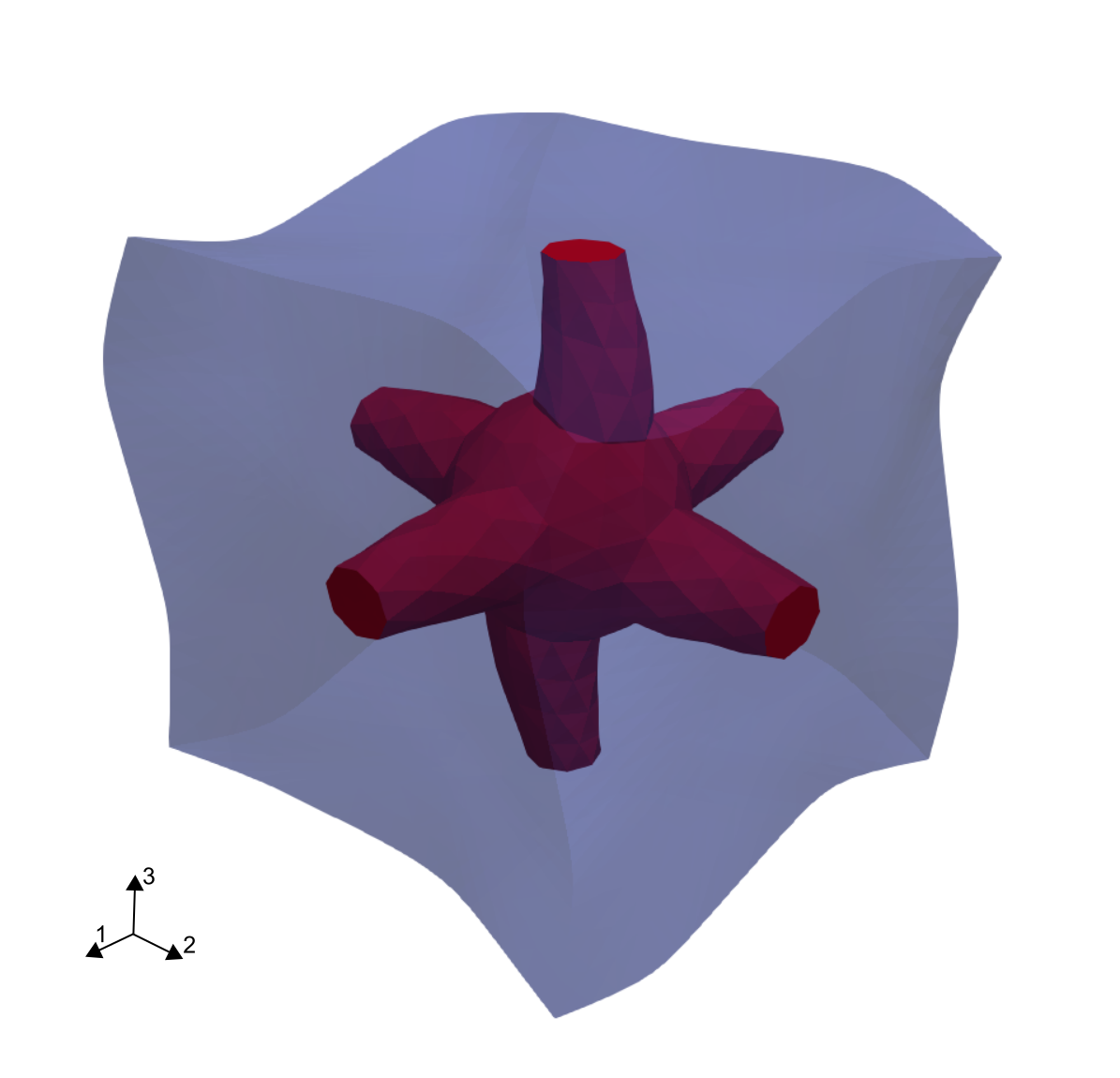}
      \label{fig-CS-a}
    \end{subfigure}
    ~
    \begin{subfigure}[t]{.43\textwidth}
      \includegraphics[trim={100pt 90pt 30pt 0},clip,width=\columnwidth]{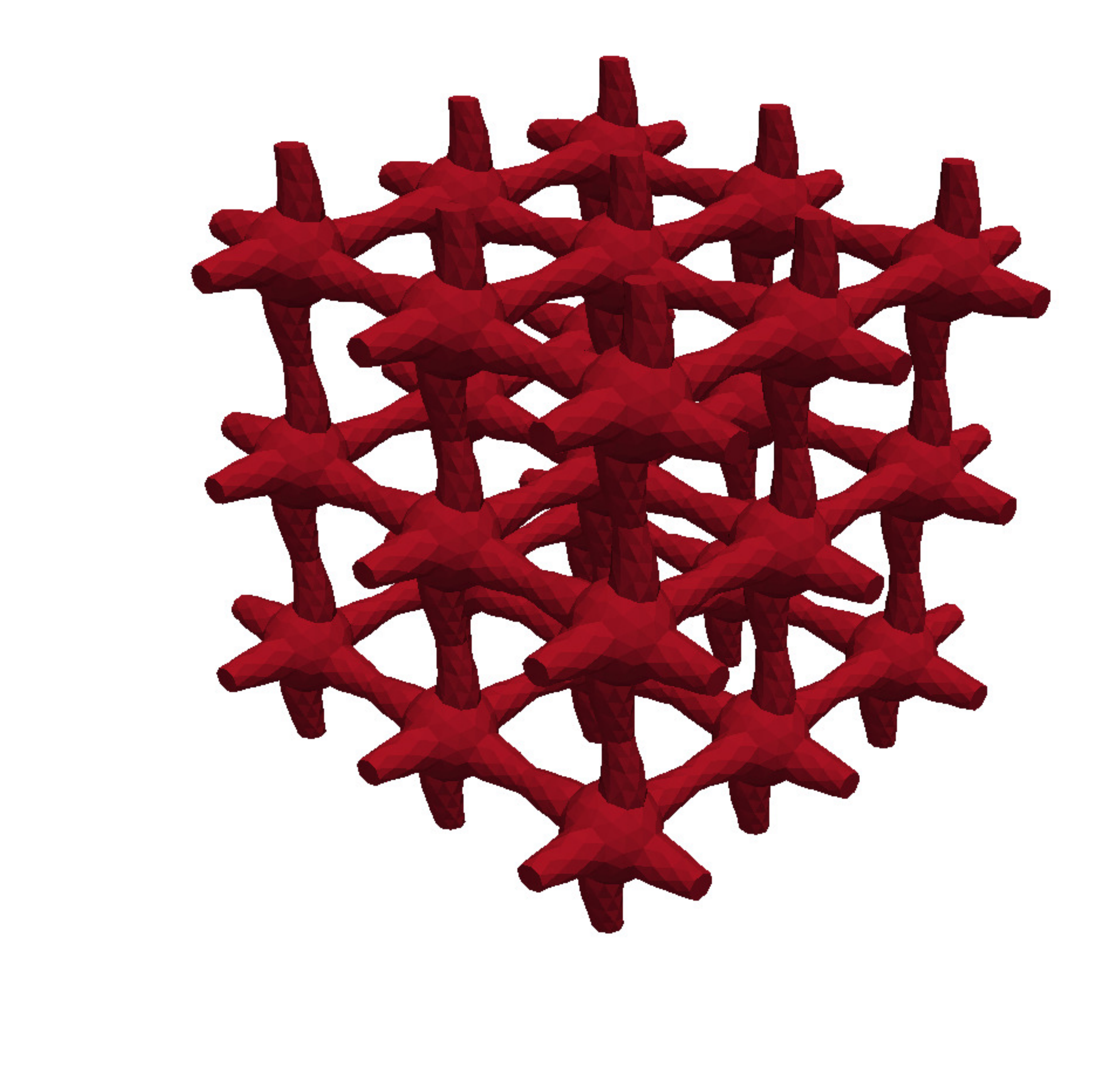}
      \label{fig-CS-b}
    \end{subfigure}
    \caption{The optimized geometry for problem \textit{C/S}. The objective (the undrained stiffness) was improved by $2.1\%$.}
    \label{fig-CS}
  \end{center}
\end{figure}


\subsection{Numerical illustrations for the 2-scale optimization}\label{sec-2scale-numres}

Next we present numerical solutions of problem \eq{eq-opg16} at a location of interest defined in a given  macroscopic domain.
First we solve the macroscopic problems \eq{eq-opg1} and \eq{eq-opg2} on a cuboid domain discretized by 15x10x2 finite elements, see Fig.~\ref{fig-macro}, where the geometry of the macroscopic specimen (domain $\Om$) and the boundary conditions are depicted.
On a part of $\Gamma_N$, surface stresses are applied, while the body is fixed at $\Gamma_D$.
We prescribe pressure at $\Gamma_p^1$ by $\bar p^1 = 1$ and at $\Gamma_p^2$ by $\bar p^2 = 0.5$.
Problem \eq{eq-opg16} also involves the Lagrange multiplier $\Lambda$ which is  a~priori unknown.
Thus, we consider different values of $\Lambda$ for which the adjoint system \eq{eq-opg11} must be solved.
From the definition of the macroscopic problem in \eq{eq-opg7}, it is readily seen that a positive $\Lambda > 0$ means too much fluid flows through $\Gamma_p^2$, while a negative $\Lambda<0$ means the opposite, the flow is insufficient.
The solution of the state problems is displayed in Fig.~\ref{fig-macro}, while solutions of the adjoint system for $\Lambda = -1$ and $\Lambda = 1$ can be seen in Fig.~\ref{fig-macro-adjoint}.

\begin{figure}[tp]
  \begin{center}
    \begin{subfigure}[t]{.5\textwidth}
      \includegraphics[width=0.9\columnwidth]{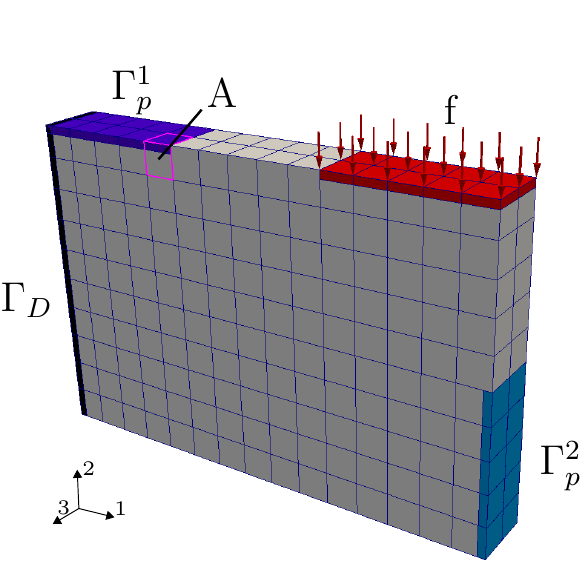}
      \label{fig-macro-setting}
    \end{subfigure}
    ~
    \begin{subfigure}[t]{.47\textwidth}
      \includegraphics[trim={100pt 100pt 100pt 100pt},clip,width=\columnwidth]{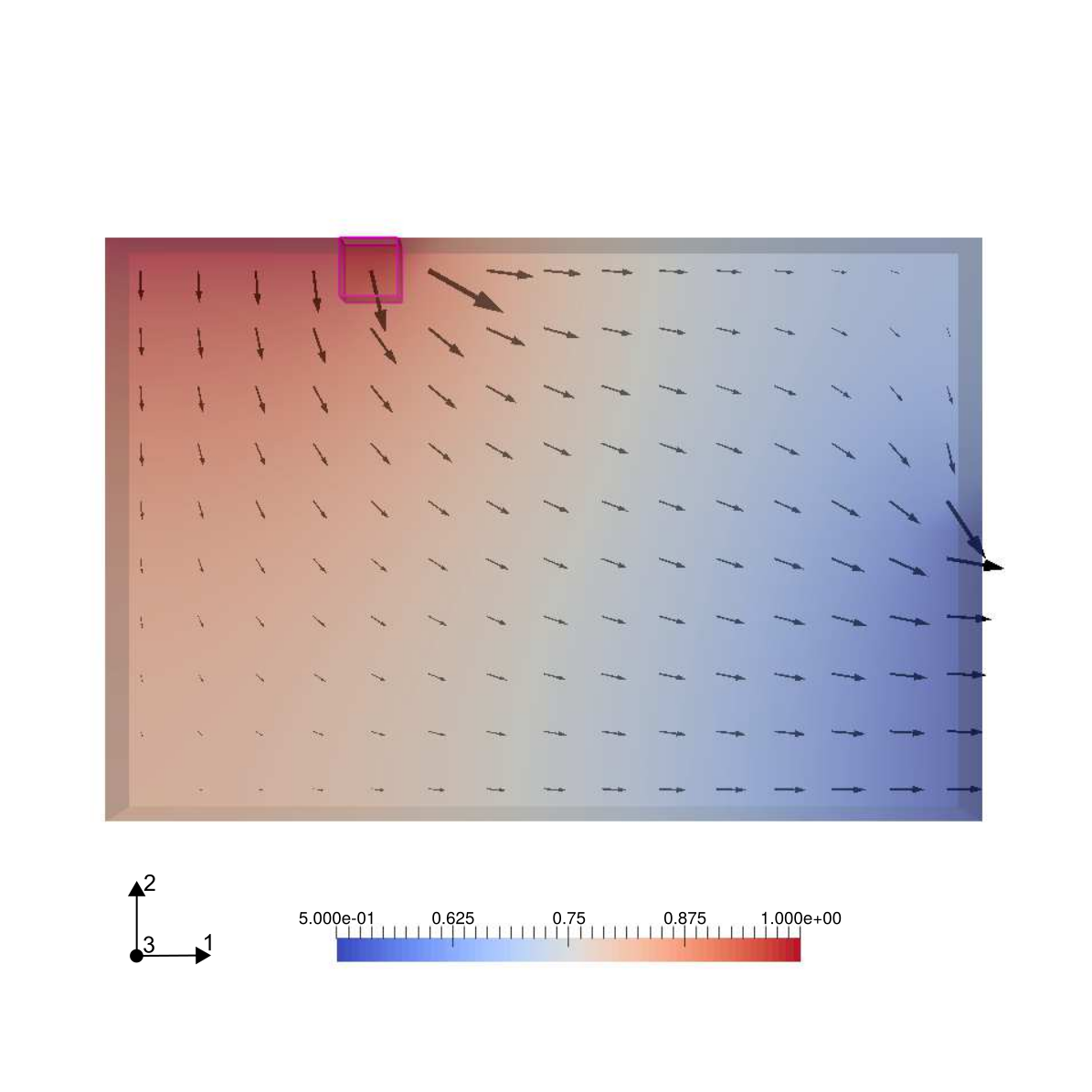}
      \label{fig-macro-state}
    \end{subfigure}
    \caption{Left: The setup of the macroscopic domain for the 2-scale examples.
On a portion of boundary $\Gamma_N$ marked by ``f'', the body is loaded by traction forces.
    Right: The solution of the state problem (front view).
    The color encodes magnitude of pressure $p$ and the arrows show directions and magnitudes of the fluid velocity.}
    \label{fig-macro}
  \end{center}
\end{figure}

\begin{figure}[tp]
  \begin{center}
    \begin{subfigure}[t]{.47\textwidth}
      \includegraphics[trim={100pt 100pt 30pt 100pt},clip,width=\columnwidth]{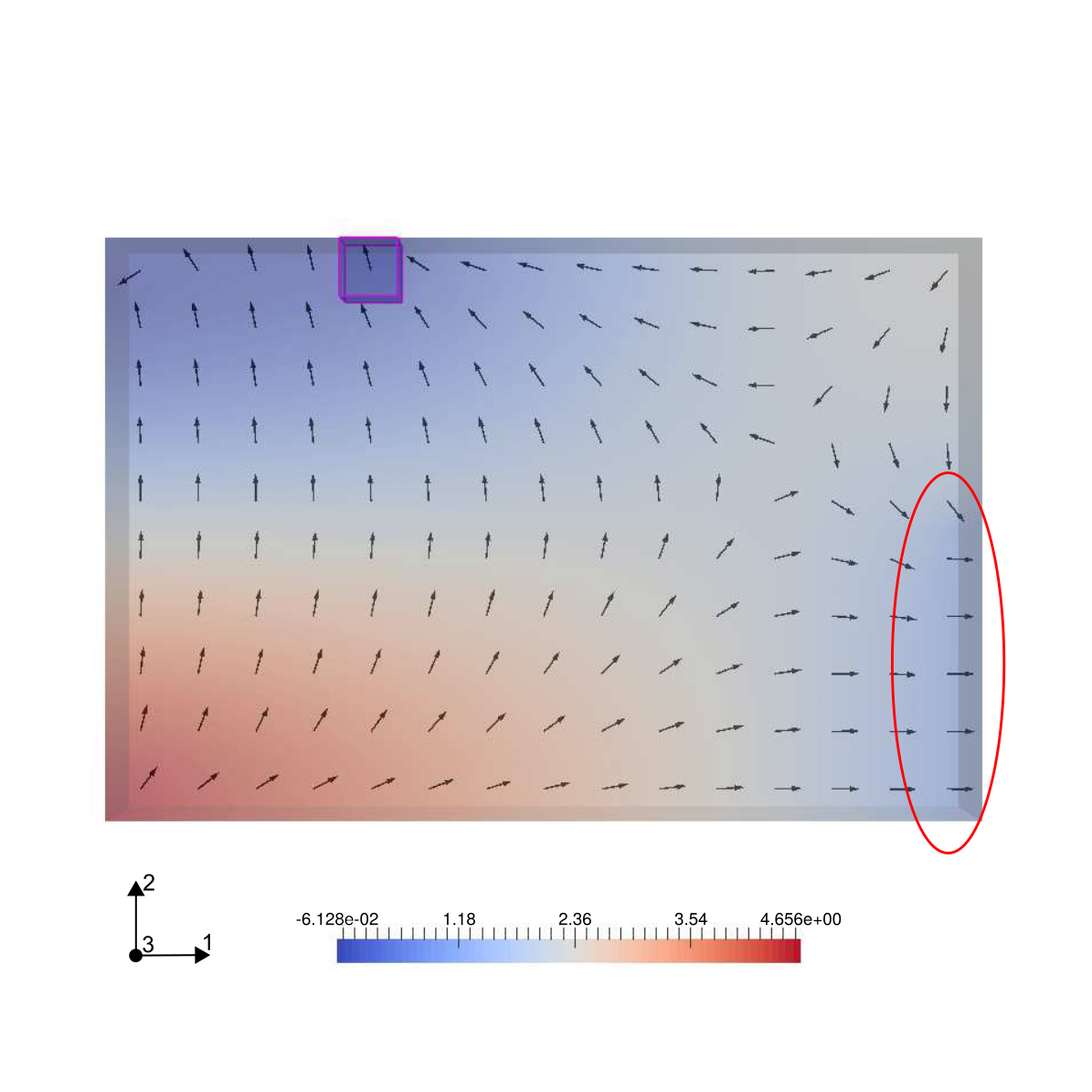}
      \label{fig-macro-adjoint-1}
    \end{subfigure}
    ~
    \begin{subfigure}[t]{.47\textwidth}
      \includegraphics[trim={100pt 100pt 30pt 100pt},clip,width=\columnwidth]{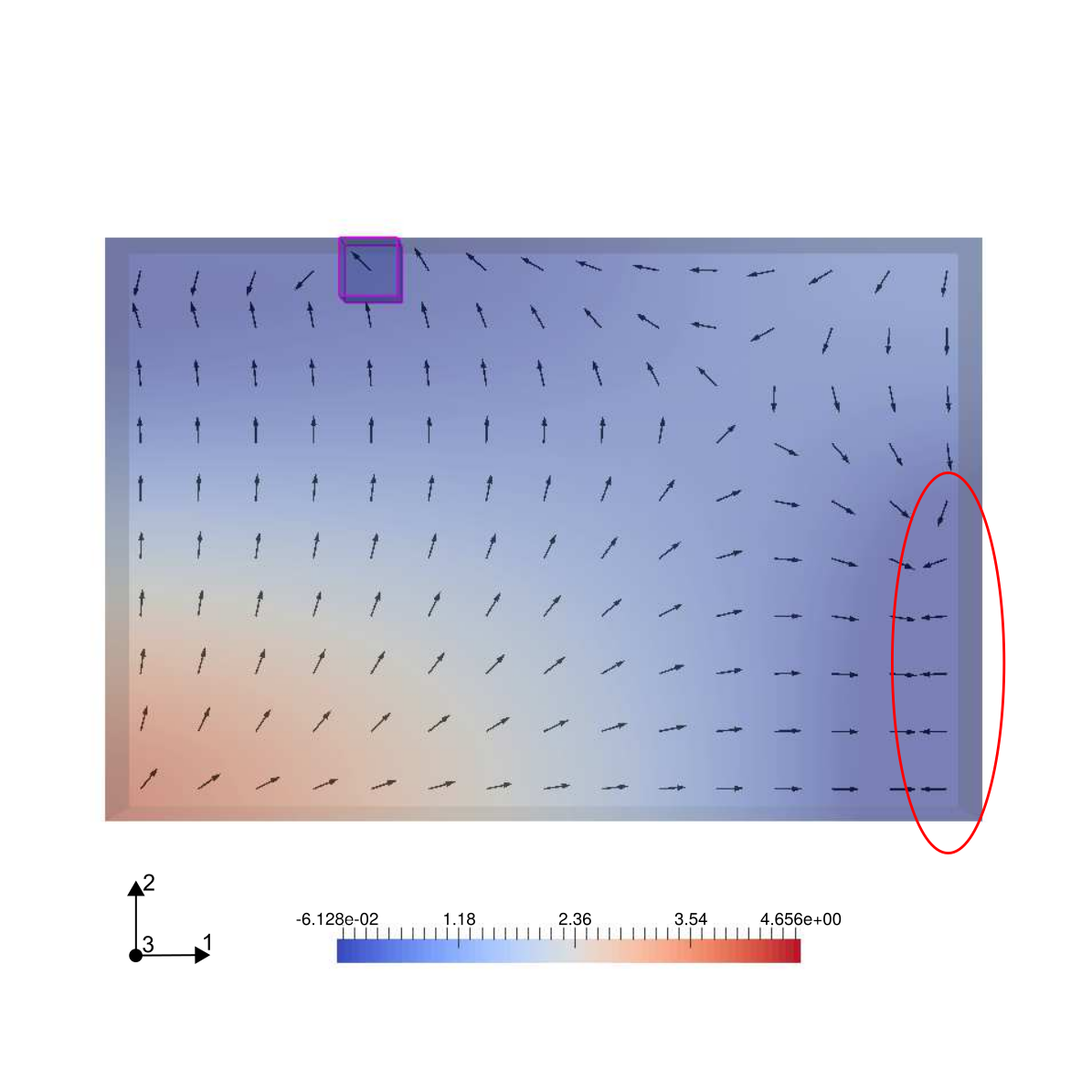}
      \label{fig-macro-adjoint+1}
    \end{subfigure}
    \caption{Left: The adjoint solution for $\Lambda = -1$.
    Right: The adjoint solution for $\Lambda = 1$.
    The color encodes magnitude of $\tilde{q}$ and the arrows show direction and magnitude of $\tilde{v}$.
    Note the opposite flow direction in the highlighted area next to $\Gamma_p^2$ which is a consequence of the sign of the multiplier $\Lambda$.}
    \label{fig-macro-adjoint}
  \end{center}
\end{figure}

We  select one finite element A located next to $\Gamma_p^1$ for which the locally periodic microstructure is optimized by virtue of the local nonlinear optimization problem \eq{eq-opg16}.
The same initial structure for the microscopic cell $Y$ is considered as in section \ref{sec-numres}, whereby the design variables are associated with the control points of a cubic spline box with three segments in each direction.
In addition, by virtue of \eq{eq-rot}, we allow for a reorientation of the cell $Y$ due to the rotation around the third axis.

For the selected finite element, the macroscopic tensors \eq{eq-opg17} are computed as the average evaluated using the values at the Gauss points.
The main eigenvector of $(\nabla P \otimes\nabla\tilde q)-\Lambda(\nabla P \otimes\nabla\tilde p)$ corresponds to the fluid velocity vector
and points into direction $\wb = [0.22,-0.97,0]$.
If we interpret $(\eeb{\ub}\otimes\eeb{\tilde\vb})$ as a fourth-order stiffness tensor, then the direction of the main stiffness points approximately into direction $\sb = [1,0,0]$.
Thus, we expect the microscopic result to have the highest permeability in direction $\wb$ and highest stiffness in direction $\sb$.
However, if we look at the convergence plot in Fig.~\ref{fig-2scale-1-convergence} we observe that the first term clearly dominates the objective function values.
As a result, the terms involving the permeability are neglected by the optimizer which focuses on increasing the stiffness of the microscopic structure in direction $\sb$.
This leads to a minimization of the channel part lumen in each cross section of $Y$ perpendicular to $\sb$ (cf. fig.~\ref{fig-2scale-1}).

\begin{figure}[tp]
  \begin{center}
    \begin{subfigure}[c]{.48\textwidth}
      \includegraphics[width=\textwidth]{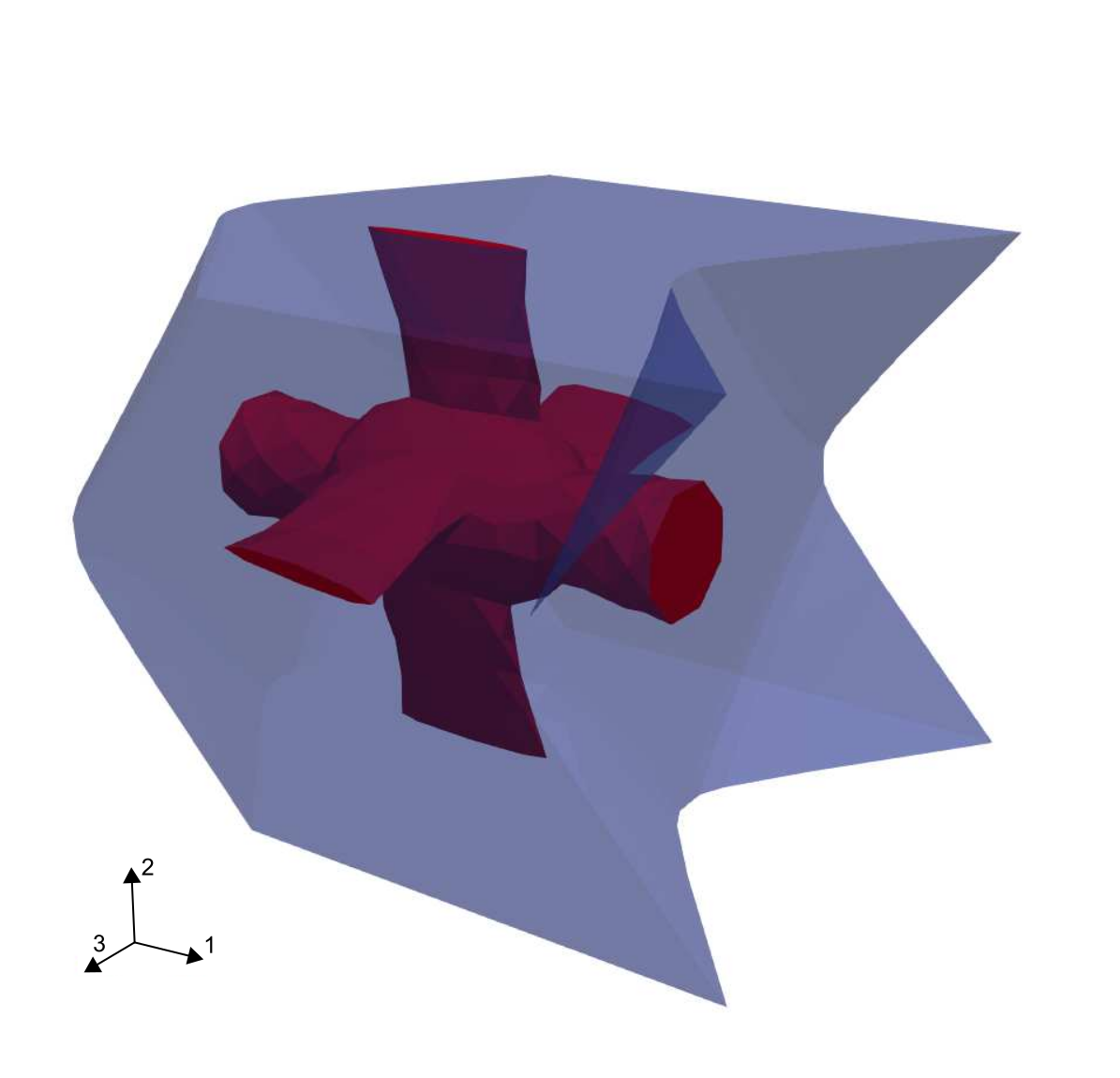}
    \end{subfigure}
    ~
    \begin{subfigure}[c]{.48\textwidth}
      \includegraphics[width=\textwidth]{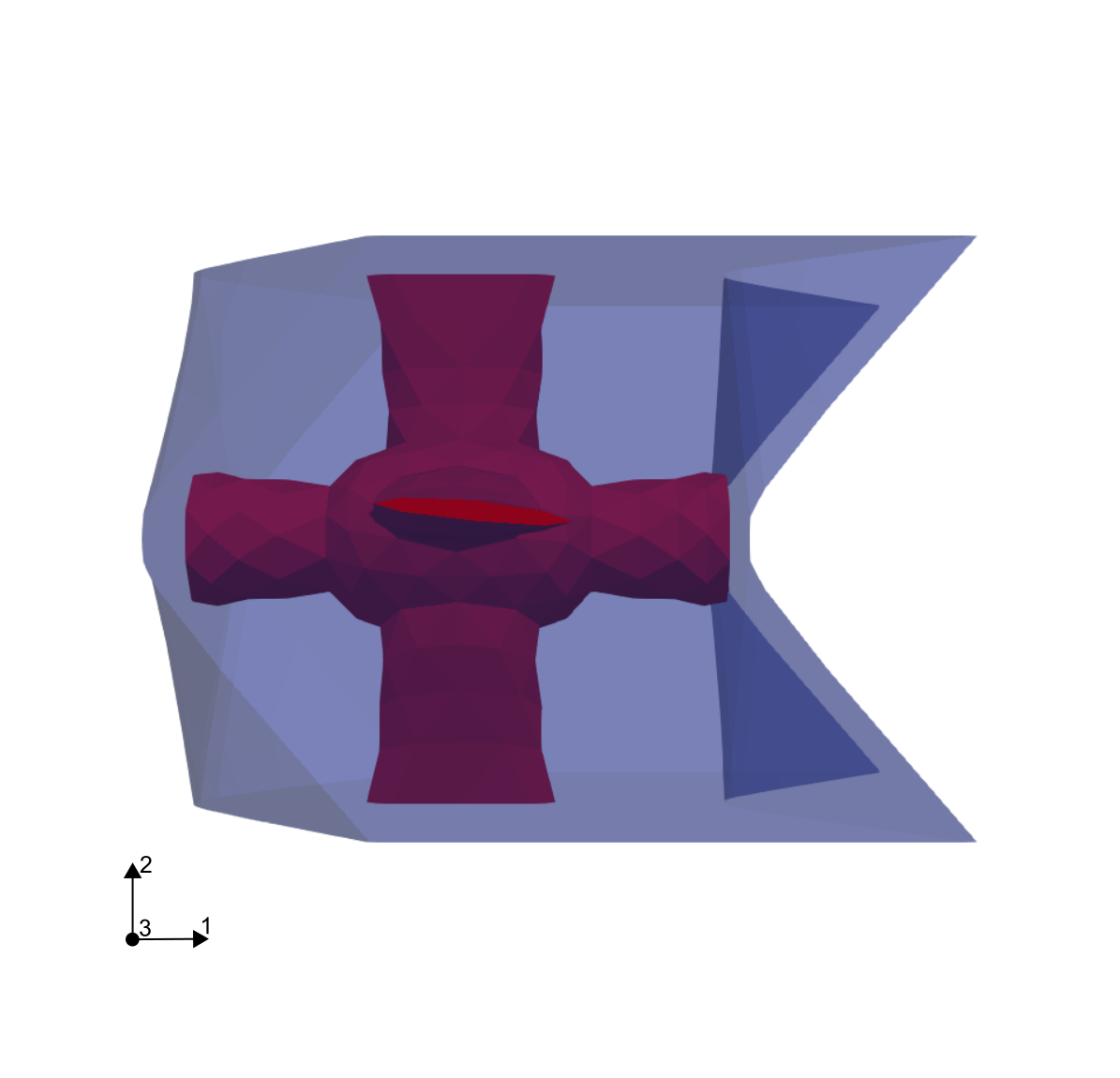}
    \end{subfigure}
    \caption{The optimized unit cell geometry for $\Lambda = -1$. The structure is rotated by $\alpha = 0.022^\circ$ around the third axis.}
    \label{fig-2scale-1}
  \end{center}
\end{figure}
\begin{figure}[tp]
  \begin{center}
    \includegraphics[width=\textwidth]{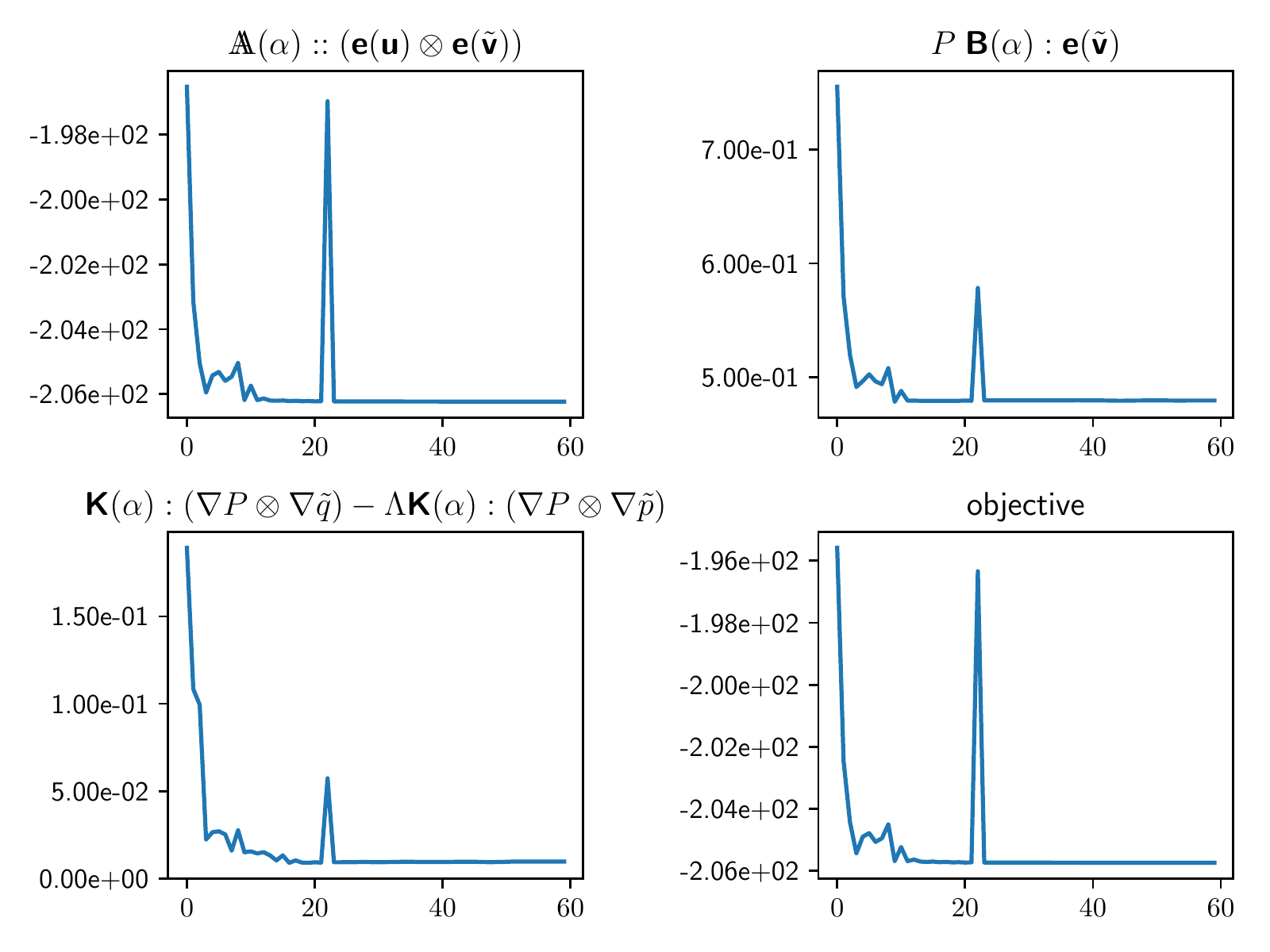}
  \end{center}
  \caption{Convergence plot for the 2-scale problem in finite element A with $\Lambda=-1$.}
  \label{fig-2scale-1-convergence}
\end{figure}

If we choose $\Lambda = -100$, still the stiffness term has most influence on the objective function, as can be seen in Fig.~\ref{fig-2scale-2-convergence}.
However, the importance of the permeability terms increases.
The optimized structure seen in Fig.~\ref{fig-2scale-2} is featured by a wide channel in vertical direction rotated by $6$ degrees, which is half the angle between $\wb$ and the second axis.
This time the stiffness tensor is almost isotropic, with only $3\%$ difference between the largest and smallest directional stiffness.
Owing to this, the rotation of the unit cell has negligible influence on the first term of the objective.
Thus a rotation is chosen which increases the permeability in direction $\wb$ and hence the overall objective.

\begin{figure}[tp]
  \begin{center}
    \begin{subfigure}[c]{.48\textwidth}
      \includegraphics[width=\textwidth]{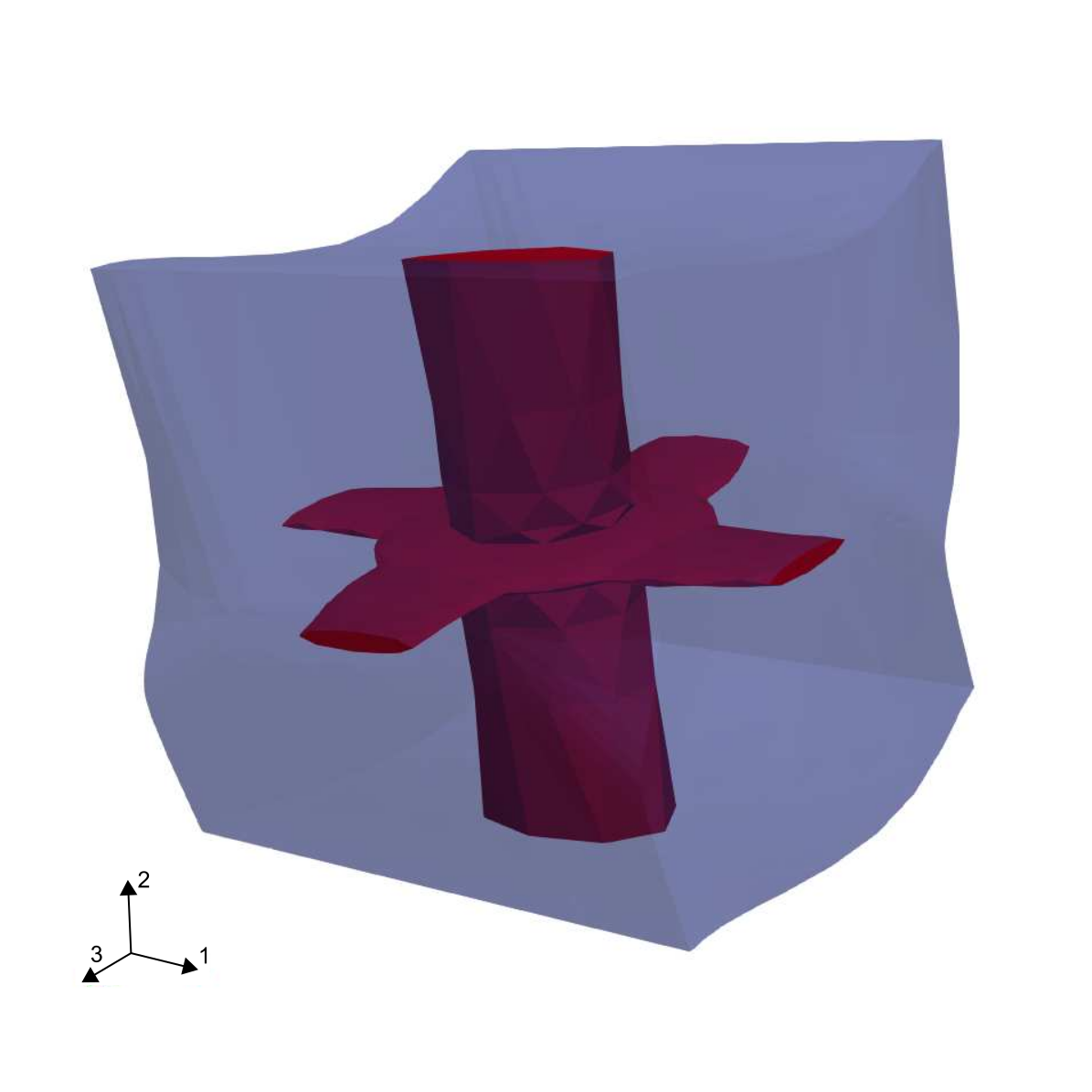}
    \end{subfigure}
    ~
    \begin{subfigure}[c]{.48\textwidth}
      \includegraphics[width=\textwidth]{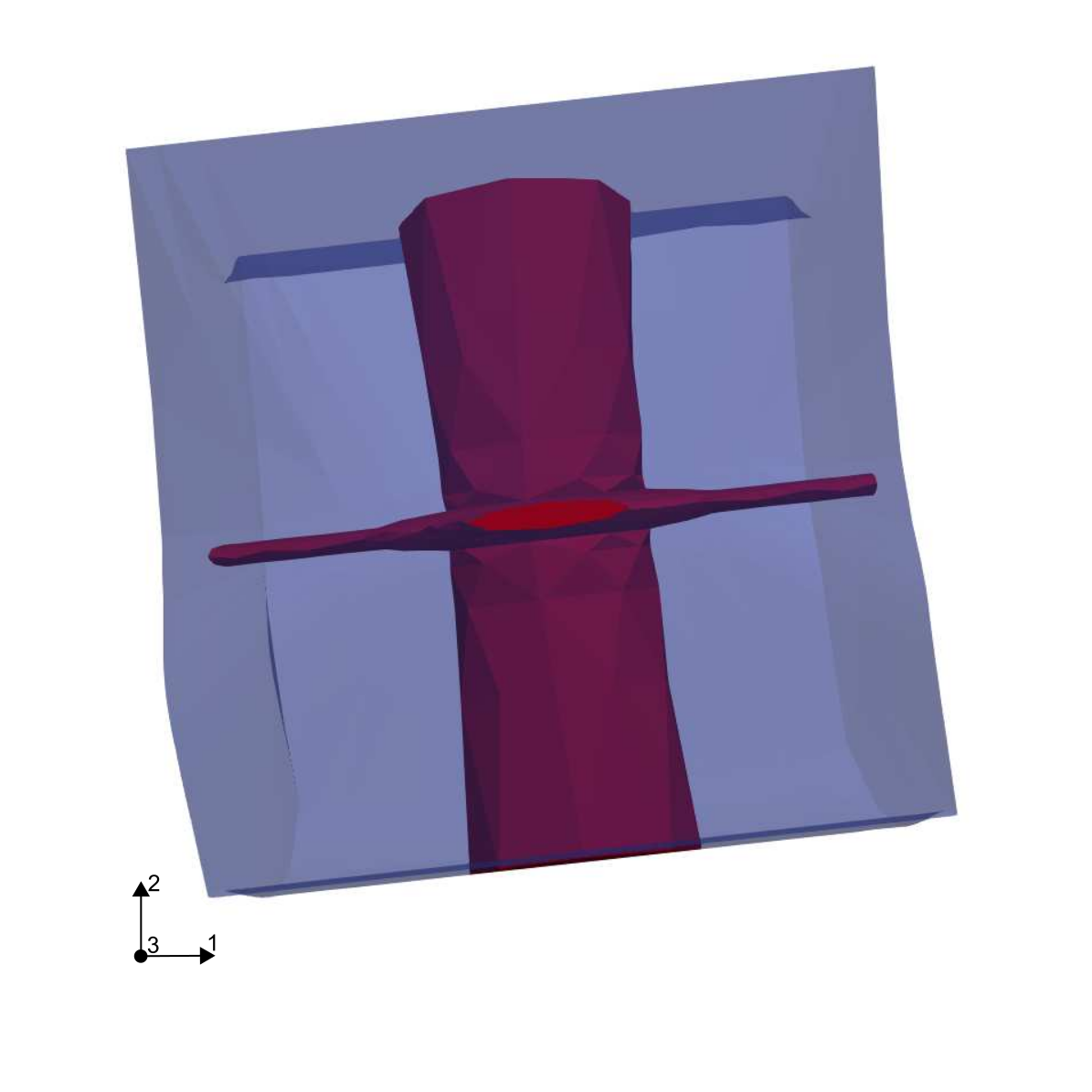}
    \end{subfigure}
    \caption{The optimized geometry for $\Lambda = -100$. The structure is rotated by $\alpha = 6.3^\circ$ around the third axis.}
    \label{fig-2scale-2}
  \end{center}
\end{figure}
\begin{figure}[tp]
  \begin{center}
    \includegraphics[width=\textwidth]{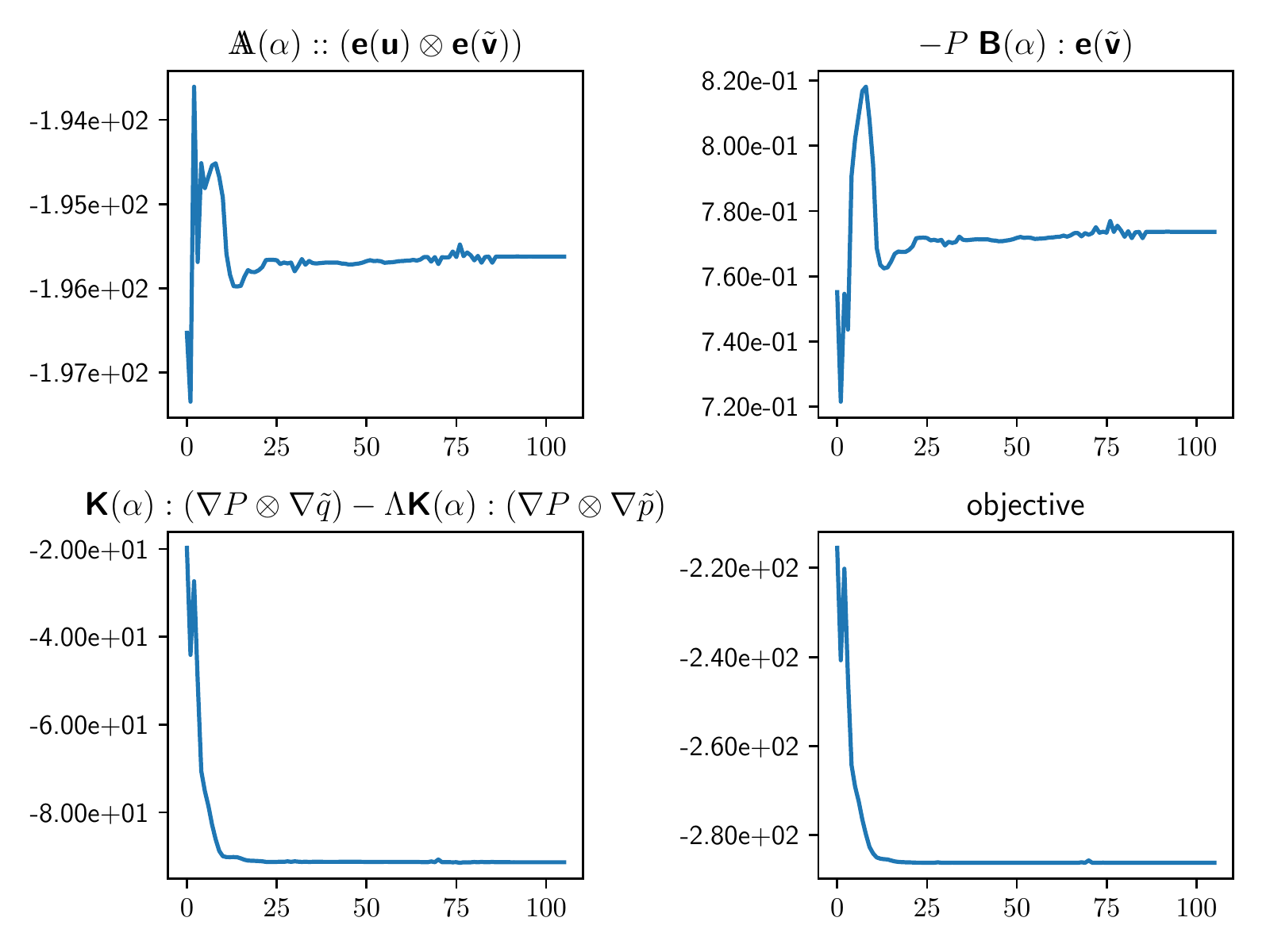}
  \end{center}
  \caption{Convergence plot for the 2-scale problem in finite element A with $\Lambda=-100$.}
  \label{fig-2scale-2-convergence}
\end{figure}

As displayed in Fig.~\ref{fig-2scale-3-convergence}, the impact of the terms involving the permeability increases with  decreasing $\Lambda$.
However, the optimized structure for $\Lambda = -1000$ is quite similar the one obtained for$\Lambda = -100$, moreover, also the rotations are similar.

\begin{figure}[tp]
  \begin{center}
    \includegraphics[width=\textwidth]{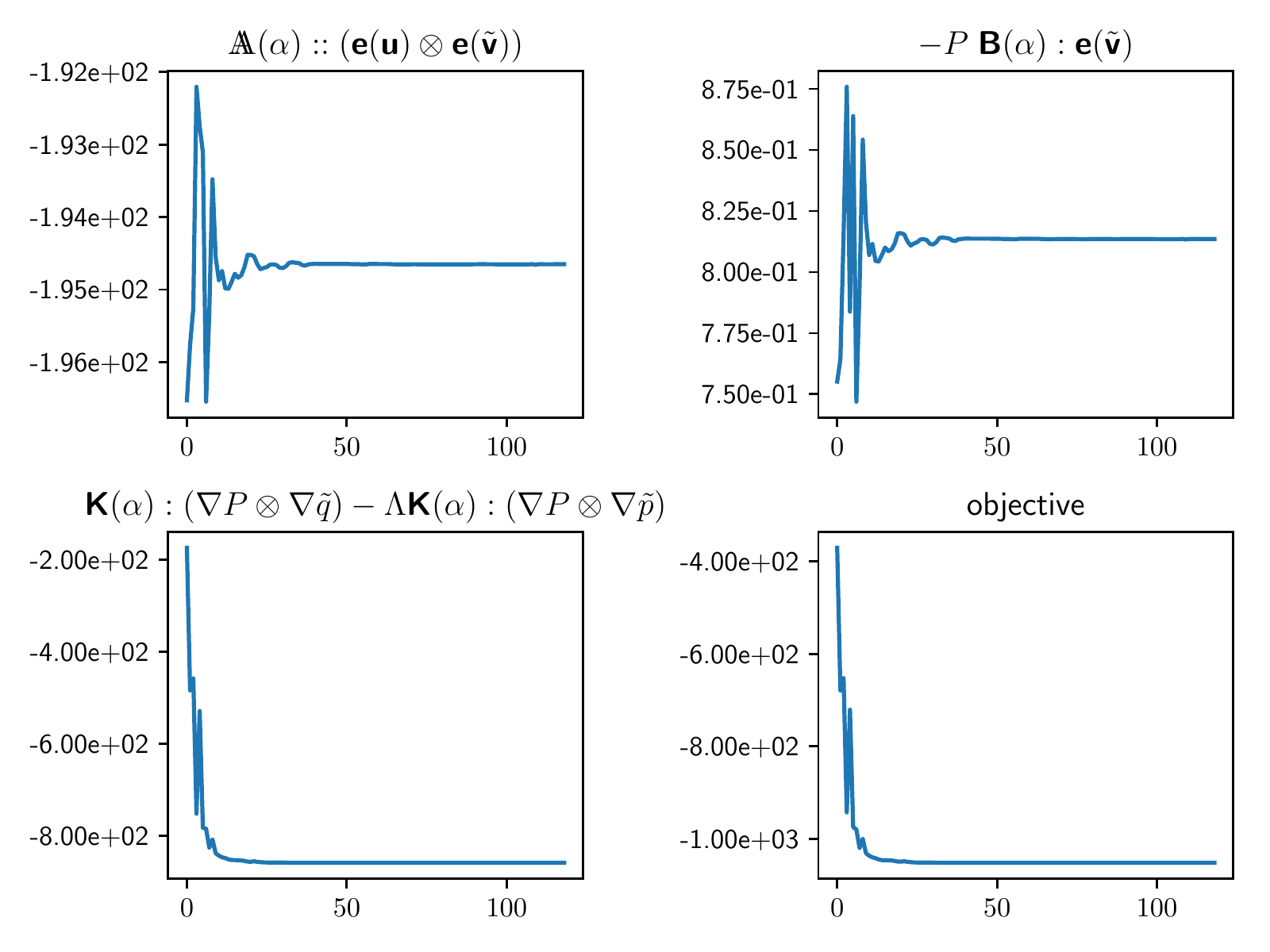}
  \end{center}
  \caption{Convergence plot for $\Lambda=-1000$}
  \label{fig-2scale-3-convergence}
\end{figure}

\section{Conclusion and outlook}\label{sec-concl}

In this study we described a methodology to optimize porous material microstructures using the shape optimization approach.
The optimality criteria are based on the effective poroelastic material coefficients.
By virtue of homogenization, these coefficients are evaluated using the characteristic responses of the representative periodic cell $Y$.
Sensitivities with respect to Biot coefficients as well as shape parameters were provided.
The shape of the fluid-filled pores was parameterized using B-spline boxes combined with only mildly restrictive regularity constraints.
As an advantage of the chosen spline  parametrization, it provides automatically the so-called design velocity fields used in the sensitivity computations which are based on the material derivative approach.
We considered optimization of the stiffness maximization subject to constraints related to the permeability and vice versa.
While in a fully symmetric setup, only moderate improvements of the chosen cost functions were possible, in an anisotropic setup significant improvements could be obtained fully exploiting the given design freedom.
These anisotropic layouts appear very naturally, when local microstructures are used to optimize some global macroscopic criterion of interest.

Following this idea, we considered a two-scale optimization problem related to the structure compliance minimization with constraints on the fluid transport.
Using the sensitivity analysis described above, a local optimization problem was derived,
which  
is essentially a linearization of the macroscopic problem with respect to the poroelastic coefficients involved in the Biot model. By virtue of the homogenization process,  these coefficients themselves are parameterized by the underlying microstructure layout; in this study, the shape of the pores is described by the spline box, cf. \cite{Andreasen2013} where the topology optimization approach was treated.
Due to its separable character, this linearized problem of the locally optimal design can be solved cell by cell.
As an example illustrating our approach, we discussed optimization results for particular areas of interest in  the macroscopic domain.
Our study should be considered as the first step towards a full two-scale optimization approach in the framework of which sequential linearizations of the given type may be used to solve the entire two-scale optimization problem.
To develop a robust two-scale optimization tool providing optimized microstructure designs in a macroscopic domain, yet some technical issues will have to be treated, such as connectivity of the pores on the boundary between subdomains of locally optimized periodic designs.

There are several extensions of the present work which are of interest in the context of smart structures and the material design. In particular, including the piezoelectric components in the microstructure will enable to control and optimize the fluid redistribution in the porous structure by the electric field; in this context, homogenization of the poro-piezoelectric materials reported in \cite{rohan-lukes-pzporel2017} can serve a suitable framework. Another important extension of the present study is by including the inertia effects in the micromodel. This will allow for optimization of  dynamic response of the porous structures to control the acoustic wave propagation and the wave dispersion.

\paragraph{Acknowledgment} This research was supported by project GACR 16-03823S and in part by projects GACR 13-00863S and LO 1506 of the Czech Ministry of Education, Youth and Sports, which enabled short visits of the first author in the NTIS, the University of West Bohemia in Pilsen.


\bibliographystyle{elsarticle-num}

\end{document}